\documentclass{amsart}

\usepackage{amssymb}
\usepackage{epsfig} 
\usepackage{graphicx}

\usepackage[all]{xy}

\def\comment#1{{\sf [#1]}}

\numberwithin{equation}{section}

\def\Z{{\mathbb Z}}
\def\Q{{\mathbb Q}}
\def\R{{\mathbb R}}
\def\C{{\mathbb C}}
\def\bB{{\mathbb B}}
\def\bD{{\mathbb D}}
\def\E{{\mathbb E}}
\def\H{{\mathbb H}}
\def\P{{\mathbb P}}
\def\U{{\mathbb U}}
\def\V{{\mathbb V}}

\def\bL{{\mathbb L}}

\def\A{{\mathcal A}}
\def\B{{\mathcal B}}
\def\F{{\mathcal F}}
\def\J{{\mathcal J}}
\def\K{{\mathcal K}}
\def\L{{\mathcal L}}
\def\M{{\mathcal M}}

\def\O{{\mathcal O}}
\def\T{{\mathcal T}}
\def\X{{\mathcal X}}

\def\cC{{\mathcal C}}
\def\cD{{\mathcal D}}
\def\cE{{\mathcal E}}
\def\cG{{\mathcal G}}
\def\cH{{\mathcal H}}

\def\cP{{\mathcal P}}
\def\cU{{\mathcal U}}
\def\cV{{\mathcal V}}

\def\g{{\mathfrak g}}
\def\u{{\mathfrak u}}

\def\sp{{\mathfrak{sp}}}

\def\bO{{\pmb{\O}}}

\def\tbar{\overline{t}}
\def\ubar{\overline{u}}
\def\vbar{\overline{v}}

\def\Fbar{\overline{\F}}
\def\Mbar{{\overline{\M}}}
\def\Hbar{{\overline{{\cH}}}}
\def\Tbar{\overline{T}}
\def\Xbar{{\overline{X}}}
\def\Ybar{\overline{Y}}

\def\Ctilde{{\widetilde{C}}}

\def\Bhat{\widehat{\B}}
\def\Bbar{\overline{\B}}
\def\Udual{\check{\U}}

\def\That{{\widehat{T}}}

\def\nutilde{\tilde{\nu}}

\def\stilde{\tilde{s}}

\def\D{{\Delta}}

\def\tauhat{\hat{\tau}}
\def\phihat{\hat{\phi}}

\def\v{{\vec{v}}}

\def\d{{\delta}}
\def\k{{\kappa}}
\def\w{{\omega}}
\def\G{{\Gamma}}

\def\dd{{\mathbf{d}}}
\def\ee{{\mathbf{e}}}

\def\Gm{{\mathbb{G}_m}}

\def\dot{{\bullet}}
\def\del{\partial}
\def\delbar{\overline{\del}}
\def\c{^c}

\def\hodge{{\sf{MHS}}}

\def\sing{{\mathrm{sing}}}

\def\blank{\phantom{x}}

\newcommand\id{\operatorname{id}}
\newcommand\ord{\operatorname{ord}}
\newcommand\Hom{\operatorname{Hom}}
\newcommand\End{\operatorname{End}}
\newcommand\Ext{\operatorname{Ext}}
\newcommand\Aut{\operatorname{Aut}}
\newcommand\Gr{\operatorname{Gr}}
\newcommand\Jac{\operatorname{Jac}}
\newcommand\Pic{\operatorname{Pic}}
\newcommand\Alb{\operatorname{Alb}}
\newcommand\Diff{\operatorname{Diff}}
\newcommand\Sp{\operatorname{Sp}}
\newcommand\GSp{\operatorname{GSp}}
\newcommand\Res{\operatorname{Res}}
\newcommand\trace{\operatorname{tr}}

\newtheorem{theorem}[equation]{Theorem}
\newtheorem{lemma}[equation]{Lemma}
\newtheorem{proposition}[equation]{Proposition}
\newtheorem{corollary}[equation]{Corollary}

\theoremstyle{definition}
\newtheorem{definition}[equation]{Definition}
\newtheorem{example}[equation]{Example}

\newtheorem{conjecture}[equation]{Conjecture}

\theoremstyle{remark}
\newtheorem{remark}[equation]{Remark}

\newtheorem{problem}[equation]{Problem}

\begin{document}

\title{Normal Functions and the Geometry of Moduli Spaces of Curves}

\author[Richard Hain]{Richard Hain}
\address{Department of Mathematics\\ Duke University\\
Durham, NC 27708-0320}

\email{hain@math.duke.edu}

\thanks{Supported in part by grants DMS-0103667 and DMS-1005675 from the
National Science Foundation.}

\date{\today}



\maketitle

\tableofcontents

\section{Introduction}

A normal function on a complex manifold $T$ is a special kind of holomorphic
section of a bundle $\J(\V) \to T$ of compact complex tori constructed from a
weight $-1$ variation of Hodge structure $\V$ over $T$. Normal functions in
their modern formulation arose in the work \cite{griffiths} of Griffiths as a
tool for understanding algebraic cycles in a complex projective manifold.

In this paper we give two examples to illustrate how normal functions might be a
useful, if unconventional, tool for understanding the geometry of moduli spaces
of curves. The first is to give a partial answer to a question of Eliashberg,
which arose in symplectic field theory \cite{sft}. Namely, we compute the class
in rational cohomology of the pullback of the $0$-section of the universal
jacobian $\J_{g,n}^c\to \M_{g,n}^c$ over the moduli space of $n$-pointed, stable
projective curves of compact type of genus $g$ along the section defined by
$$
F_\dd : [C;x_1,\dots,x_n] \mapsto \sum_{j=1}^n d_j [x_j] \in \Jac C,
$$
where $\dd=(d_1,\dots,d_n)\in \Z^n$ satisfies $\sum_{j=1}^n d_j = 0$. This
section is a normal function.

The second application is to give an alternative and complementary approach to
the slope inequalities of the type discovered by Moriwaki
\cite{moriwaki,moriwaki:new} such as his result that the divisor class
$$
M :=
(8g+4)\lambda_1 - g\delta_0 - 4\sum_{h=1}^{\lfloor g/2\rfloor} h(g-h)\delta_h
$$
on $\Mbar_g$ has non-negative degree on all complete curves in $\Mbar_g$ that do
not lie in the boundary divisor $\D := \Mbar_g-\M_g$.\footnote{A weaker version
of this inequality had been proved previously by Cornalba and Harris in
\cite{cornalba-harris}.} (Here the $\d_h$ ($0\le h \le g/2$) denote the classes
of the components of $\D$.) This alternative approach leads to actual and
conjectural strengthenings of his inequalities: we show that the Moriwaki
divisor $M$ has non-negative degree on all complete curves in $\M_g^c$ and
conjecture that $M$ has non-negative degree on all complete curves in $\Mbar_g$
that do not lie in the boundary divisor $\D_0:= \Mbar_g - \M_g^c$.

Each normal function section $\nu$ of $J(\V) \to T$ determines a class $c(\nu)$
in $H^1(T,\V)$. When $H^0(T,\V_\Q)$ vanishes, the normal function $\nu$ is
determined, mod torsion, by its characteristic class $c(\nu)$. In such cases,
there is a rigid relationship between normal functions and cohomology. When $T =
\M_{g,n}^c$, $g\ge 3$, and $\V$ is a variation of Hodge structure corresponding
to a non-trivial rational representation of $\Sp_g$ that does not contain the
trivial representation, the group of normal function sections of $J(\V)$ is
finitely generated and is known modulo torsion, \cite[\S8]{hain:normal}. This
result is recalled in Appendix~\ref{sec:vmhs}. When $\V$ corresponds to the
fundamental representation of $\Sp_g$, $J(\V)$ is the universal jacobian
$\J_{g,n}^c$ over $\M_{g,n}^c$ and the class of the pullback $F_\dd^\ast\eta_g$
of the zero section of $\J_{g,n}^c$ can be expressed in terms of the classes of
certain basic normal functions defined on $\M_{g,n}^c$.

All variations of Hodge structure $\V$ of geometric origin have a {\em
polarization}; that is, an invariant inner product $S : \V^{\otimes 2} \to \Q$
that satisfies the Riemann-Hodge bilinear relations on each fiber. When $V$ has
odd weight, the polarization is skew symmetric. Each invariant, skew-symmetric
inner product $S : \V^{\otimes 2} \to \Q$ gives rise to a class $S\circ
c(\nu)^2$ in $H^2(T,\Q)$. It is the image of $c(\nu)^{\otimes 2}$ under the
composition of the cup product with the map induced by $S$:
$$
\xymatrix{
H^1(T,\V)^{\otimes 2} \ar[r]^{\smile} & H^2(T,\V^{\otimes 2})
\ar[r]^{S_\ast} & H^2(T,\Q).
}
$$
The class $S\circ c(\nu)^2$ has a natural de~Rham representative which is a
non-negative $(1,1)$-form when $S$ is a polarization. Moriwaki's inequality for
complete curves in $\M_g^c$ is an immediate consequence of this semi-positivity
and the fact that the class of the Moriwaki divisor equals the square $S\circ
c(\nu)^2$ of the class of the most fundamental normal function over $\M_g$ ---
viz., the normal function associated to the cycle $C-C^-$ in $\Jac C$ that was
first studied by Ceresa \cite{ceresa}.

When $\V$ is a weight $-1$ polarized variation of Hodge structure over $T$,
there is a naturally metrized line bundle $\B$ over $J(\V)$, which is called the
{\em biextension line bundle}. The curvature of its pullback along a normal
function $\nu : T \to J(\V)$ is the natural de~Rham representative of $S\circ
c(\nu)^2$. This line bundle extends naturally to any compactification $\Tbar$ of
$T$, even if $\V$ (and hence $\nu$) does not extend to $\Tbar$. This extension
is characterized by the property that the metric extends across the codimension
1 strata of $\Tbar-T$. Surprisingly, this extension is {\em not} natural under
pullback to a smooth variety as the metric on the extended line bundle may be
singular on strata of codimension $\ge 2$ of $\Tbar-T$. This curious phenomenon
is called {\em height jumping}. Height jumping and its relevance to refined
slope inequalities is discussed in Section~\ref{sec:moriwaki}.

The classes $c(\nu)$ form part of a larger structure. When $T=\M_{g,n}$ they
should be regarded as twisted tautological cohomology classes. To explain this,
we need to introduce a certain graded commutative algebra associated to
$\M_{g,n}$. Denote the coordinate ring of the symplectic group $\Sp_g$ by
$\O(\Sp_g)$. Left and right multiplication induce commuting left and right 
actions of the mapping class group $\pi_1(\M_{g,n},\ast)$ on it via the standard
representation $\pi_1(\M_{g,n},x_o) \to \Sp_g(\Q)$. Using the right action, one
obtains a local system $\bO$ of $\Q$-algebras over $\M_{g,n}$. Since $\bO$ is a
local system of commutative $\Q$-algebras, its cohomology
$$
A_{g,n}^\dot := H^\dot(\M_{g,n},\bO)
$$
is a graded commutative $\Q$-algebra. The left $\Sp_g$-action on $\bO$ gives it
the structure of  a graded commutative algebra in the category of
$\Sp_g$-modules.

The algebraic analogue of the Peter-Weyl Theorem implies that there is an
$\Sp_g\times\Sp_g$-equivariant isomorphism
$$
\O(\Sp_g) \cong \bigoplus_\lambda \End(V_\lambda)^\ast
\cong \bigoplus_\lambda V_\lambda\boxtimes V_\lambda^\ast,
$$
where $\{V_\lambda\}$ is a set of representatives of the isomorphism classes of
irreducible $\Sp_g$-modules.  There is thus an isomorphism
$$
A_{g,n}^\dot \cong
\bigoplus_\lambda H^\dot(\M_{g,n},\V_\lambda)\otimes V_\lambda^\ast
$$
where $\V_\lambda$ denotes the local system over $\M_{g,n}$ that
corresponds to $V_\lambda$.

The classes $c(\nu)$ are more fundamental than their squares $S\circ c(\nu)^2$,
which are known to be tautological classes. For this reason, we define the
tautological subalgebra $T_{g,n}^\dot$ of $A_{g,n}^\dot$ to be the graded
subalgebra generated by the classes $c(\nu)\otimes V_\lambda^\ast$ of the normal
function sections $\nu$ of the $J(\V_\lambda)$.\footnote{Each $\V_\lambda$ is
the local system that underlies a polarized variation of Hodge structure over
$\M_{g,n}$. It is unique up to Tate twist. The only $\V_\lambda$ that admit
non-torsion normal functions have weight $-1$.} It is finitely generated. The
classification of normal functions \cite{hain:normal} over $\M_{g,n}$, and the
work of Kawazumi and Morita \cite{kawa-morita} imply that the ring of
$\Sp_g$-invariants $(T_{g,n}^\dot)^{\Sp_g}$ is Faber's tautological ring
$R_{g,n}$ (in cohomology) \cite{faber} of $\M_{g,n}$. The computations of Morita
\cite{morita:taut}, Kawazumi and Morita \cite{kawa-morita}, and those of this
paper, may be regarded as computations in $T_{g,n}^\dot$. For this reason, we
propose that the ring $T_{g,n}^\dot$ is more fundamental than its
$\Sp_g$-invariant part $R_{g,n}$. It would be interesting to define and study a
Chow analogue of $T_{g,n}^\dot$. The significance of this algebra and its
relation to normal functions is discussed in Appendix~\ref{sec:big_picture}.

\bigskip

\noindent{\it Advice to the reader:} Although normal functions have long been a
part of algebraic geometry (examples were first considered by Poincar\'e), they
are not currently part of the standard repertoire of modern algebraic geometry.
Their modern definition (Definition~\ref{def:normal}), in terms of extensions of
variations of Hodge structure, requires an understanding of variations of
(mixed) Hodge structure. However, if the reader is prepared to believe that the
local systems associated to locally topologically trivial families of algebraic
varieties are motivic, and so are variations of mixed Hodge structure, then the
definition should be natural.

We assume the reader is familiar with the basic definitions and constructions of
Hodge theory. In particular, the reader should know the definition of Hodge
structures, mixed Hodge structures, and variations of Hodge structure. The book
\cite{peters-steenbrink} by Peters and Steenbrink is a good source of basic
material on these topics. The paper \cite{hain:cdm} contains a brief exposition
of Schmid's work \cite{schmid} on the asymptotic properties of variations of
Hodge structure that emphasizes the case of degenerations of curves. Finally,
the recent survey of normal functions \cite{kerr-pearlstein} by Kerr and
Pearlstein should be a useful supplement, although its emphasis is quite
different from that of this article.

\bigskip

\noindent{\it Acknowledgments:} I am grateful to Yasha Eliashberg for posing his
question to me in 2001 and to Gavril Farkas for his interest, which resulted in
this paper seeing the light of day. Theorem~\ref{thm:eliashberg} was proved
during visits to the University of Sydney and the Universit\'e de Nice during
the author's sabbatical in 2002--03. Many thanks to both institutions for their
support, and to my respective hosts, Gus Lehrer and Arnaud Beauville, for their
hospitality. I am especially grateful to Samuel Grushevsky, Robin de Jong and
Dmitry Zharkov for their interest in this work and for their constructive
comments and corrections. Sam and Dmitry pointed out an error in the coefficient
of $\d_h^P$ in Theorem~\ref{thm:eliashberg}; Sam and Robin isolated the error,
which was a missing term in the formula Theorem~\ref{phi2}.

I would also like to thank Gregory Pearlstein for his numerous constructive
comments on the manuscript, and also Renzo Cavalieri for communicating his
related results \cite{cavalieri-marcus} with Steffen Marcus. The section on the
genus 1 case of Eliashberg's problem was written as a result of correspondence
with him. Finally, I would like to thank the referee for helpful comments.

\medskip

\section{Notation and Conventions}

All varieties (and stacks) will be defined over the complex numbers. Denote the
moduli space of stable $n$-pointed curves of genus $g$ by $\Mbar_{g,n}$. This is
defined when $2g-2+n>0$ and will be viewed as a stack or as a complex analytic
orbifold. As such, it is smooth. Denote the Zariski open subset corresponding to
the set of $n$-pointed smooth curves by $\M_{g,n}$ and by $\M_{g,n}^c$ the
Zariski open subset consisting of $n$-pointed curves of compact type. Note that
$\M_{g,n}^c = \Mbar_{g,n}-\D_0$, where $\D_0$ denotes the boundary divisor of
$\Mbar_{g,n}$ whose generic point is an irreducible, geometrically connected
curve with one node.

The moduli stack of principally polarized abelian varieties of dimension $g$
will be denoted by $\A_g$. The universal curve of compact type will be denoted
by $\cC_g^c \to \M_g^c$ and its restriction to $\M_g$ by $\cC_g$. All of these
moduli spaces are globally the quotient of a smooth variety by a finite group,
\cite{looijenga,boggi-pikaart}.

Vector bundles, variations of (mixed) Hodge structure, etc on a stack $T$ that
is the quotient of a smooth variety $S$ by a finite group $G$, are $G$-invariant
bundles, variations of (mixed) Hodge structure, etc, over $S$. Since all stacks
that occur in this paper are of this form, we will not distinguish between
stacks and varieties, as working on one of these stacks is working equivariantly
on a finite cover that is a variety.

The category of $\Z$-mixed Hodge structures by $\hodge$. For $d\in\Z$, the Hodge
structure of type $(-d,-d)$ whose underlying lattice is isomorphic to $\Z$ will
be denoted by $\Z(d)$. We shall denote by $\G V$ the set $\Hom_\hodge(\Z(0),V)$
of Hodge classes of type $(0,0)$ of the mixed Hodge structure $V$ . The category
of admissible variations of mixed Hodge structure over a smooth variety $X$ will
be denoted by $\hodge(X)$.

All cohomology groups will be with $\Q$ coefficients unless otherwise stated.
Similarly, the Chow group of codimension $d$ cycles on a stack $X$, {\em
tensored with $\Q$}, will be denoted by $CH^d(X)$.

\section{Eliashberg's Problem}
\label{sec:eliashberg}

To motivate the discussion of normal functions and related topics in subsequent
sections, we begin with a brief discussion of the universal jacobian and
Eliashberg's problem. Some readers may prefer to begin with Sections~\ref{tori}
and \ref{sec:normal}. Recall that a stable curve $C$ of genus $g$ is of {\em
compact type} if its dual graph is a tree. This condition is equivalent to the
condition that its jacobian $\Jac C := \Pic^0 C$ be an abelian variety.

\subsection{The universal jacobian} We begin with a review of the transcendental
construction of the jacobian of the universal curve over $\M_g$. This is a
special case of Griffiths' construction of families of intermediate jacobians
and  normal functions, which are reviewed in Section~\ref{sec:normal}.

First recall the transcendental construction of the jacobian of a smooth
projective curve $C$, which we recast in the language of Hodge structures. It
will be generalized in Section~\ref{sec:normal} where Griffiths intermediate
jacobians are introduced. The Hodge Theorem implies that
\begin{equation}
\label{eqn:decomp}
H^1(C,\C) \cong H^{1,0}(C) \oplus H^{0,1}(C)
\end{equation}
where $H^{1,0}(C)$ denotes the space $H^0(C,\Omega_C^1)$ of holomorphic 1-forms
on $C$ and $H^{0,1}(C)$ its complex conjugate, the space of anti-holomorphic
1-forms. The first integral cohomology group $H^1(C,\Z)$ endowed with the
decomposition (\ref{eqn:decomp}) is the prototypical Hodge structure of weight
1. Its dual
$$
H_1(C,\C) = H^{-1,0}(C) \oplus H^{0,-1}(C)
$$
is a Hodge structure of weight $-1$, where $H^{-p,-q}(C)$ is defined to be the
dual of $H^{p,q}(C)$. The Hodge filtration
$$
H_1(C,\C) = F^{-1}H_1(C) \supset F^0H_1(C) \supset F^1H_1(C) = 0.
$$ 
of $H_1(C)$ is defined by
$$
F^p H_1(C) = \bigoplus_{\substack{s\ge p\cr s+t=-1}} H^{s,t}(C).
$$
The projection onto $H^{1,0}(C)$ induces an isomorphism
$$
H_1(C,\C)/F^0 \cong H^{-1,0}(C) \cong
H^0(C,\Omega^1_C)^\ast := \Hom_\C(H^0(C,\Omega^1_C),\C).
$$
The composite
$$
H_1(C,\Z) \hookrightarrow H_1(C,\C) \to H_1(C,\C)/F^0
\cong H^0(C,\Omega^1_C)^\ast
$$
is the map that takes the homology class of the 1-cycle $\gamma$ to the
functional
$$
\int_\gamma := \Big\{\omega \mapsto \int_\gamma \omega\Big\}
\in H^0(C,\Omega_C^1)^\ast. 
$$
It is injective and its image is a lattice.  The jacobian of $C$ is the
quotient
$$
\Jac C := H^0(C,\Omega^1_C)^\ast/H_1(C,\Z) \cong
H_1(C,\C)/(H_1(C,\Z) + F^0H_1(C)).
$$
Every divisor $D$ of degree $0$ on $C$ can be written as the boundary
$D=\partial \gamma$ of a real $1$-chain $\gamma$. The Abel-Jacobi mapping
$$
\{\text{divisors of degree 0 on $C$}\}/\text{rational equivalence} \to
\Jac C
$$
is defined by taking the divisor class of the boundary of the 1-chain $\gamma$
to the functional $\int_\gamma$. Abel's Theorem implies that it is a group
isomorphism. This construction works equally well when $C$ is a curve of compact
type. In this case, the Hodge structure on $H_1(C)$ is the direct sum of the
Hodge structures of its irreducible components. This construction can also be
carried out for families of complete curves, where each fiber is either smooth
or of compact type. Below we carry this out for the universal curve of compact
type. The family of Hodge structures associated to such a family is an example
of a {\em variation of Hodge structure}.

To the universal curve $\pi : \cC_g\c \to \M_g\c$ of compact type we associate
the variation of Hodge structure
$$
\H := R^1 \pi_\ast \Z
$$
and the corresponding holomorphic vector bundle $\cH := \H \otimes_\Z
\O_{\M_g\c}$. The fiber of $\cH$ over the moduli point of $C$ is $H^1(C,\C)$.
This bundle has a flat holomorphic connection $\nabla$. The local monodromy
transformations about the divisor $\Delta_0:=\Mbar_{g,n}-\M_{g,n}\c$ are given
by the Picard-Lefschetz formula and are therefore unipotent. Consequently, $\cH$
has a canonical extension (in the sense of Deligne \cite{deligne:ode}) to a
vector bundle $\Hbar$ over $\Mbar_g$.\footnote{See \cite{hain:cdm} for a concise
exposition.} It is characterized by the property that the connection is regular
and that its residue at each smooth point of $\Delta_0$ is nilpotent. Since the
monodromy of $\H$ about $\D_0$ is non-trivial, the local system $\H$ does not
extend across $\D_0$. Consequently, $\M_{g,n}^c$ is the maximal Zariski open
subset of $\Mbar_{g,n}$ to which $\H$ extends.

The {\em Hodge bundle} $\F := \F^1\cH$ is the sub-bundle of $\cH$ whose fiber
over the moduli point of $C$ is $H^{1,0}(C)$. It is holomorphic and extends,
by a result of Deligne \cite{deligne:ode}, to a holomorphic sub-bundle of
$\Fbar^1$ of $\Hbar$.\footnote{The fiber $\F_C$ of the Hodge bundle over the
stable curve $C$ can be described as follows. Denote the normalization of $C$ by
$\nu : \Ctilde \to C$. Let $D\subset \Ctilde$ be the inverse image of the double
points of $C$. Then $\cH_C$ is the subset of $H^0(\Ctilde,\Omega^1_\Ctilde(D))$
consisting of those $w$ such that $\Res_P w + \Res_Q w = 0$ whenever $P\neq Q$
and $\nu(P) = \nu(Q)$. It is naturally isomorphic to $F^1H^1(C_\v)$ where
$H^1(C_\v)$ denotes the limit MHS on the first order smoothing $C_\v$ of $C$
associated to a tangent vector $\v$ of $\Mbar_g$ at $[C]$ which is not tangent
to the boundary divisor.}  (See \cite[\S4]{hain:cdm} for an exposition.) There
is a natural projection $\check\F \to \J_g$ from the dual of the Hodge bundle to
the universal jacobian, which is a covering map on each fiber. The kernel of the
projection
$$
\check\F_C \to \J_{g,C}
$$
at the moduli point of the stable curve $C$ is the image of $H_1(C-C^\sing,\Z)$
under the integration map $H_1(C-C^\sing,\Z)\rightarrow \check\F_C$. The
restriction of $\J_g$ to $\Mbar_g - \D_0^\sing$ is a Hausdorff complex analytic
orbifold, a fact which follows, for example, from \cite[Prop.~2.9]{zucker}. It
is the analytic orbifold associated to the restriction of the universal $\Pic^0$
stack over $\Mbar_g$ to $\Mbar_g - \D_0^\sing$, which is constructed in
\cite{caporaso}.

Observe that the fiber of $\J_g \to \Mbar_g$ over the moduli point of a stable
curve $C$ is an abelian variety if and only if $C$ is of compact type. From the
construction, it is clear that the normal bundle of the zero section of $\J_g$
is the dual $\check\F$ of the Hodge bundle.

\subsection{Eliashberg's question}

Suppose that $2g-2+n>0$. Given an integer vector $\dd=(d_1,\dots, d_n)$ with
$\sum_j d_j=0$, we have the {\em rational} section $F_\dd$
$$
\xymatrix{
	& \J_g \ar[d] \cr
\Mbar_{g,n}\ar[r]\ar@{.>}[ur]^{F_\dd} & \Mbar_g
}
$$
of the universal jacobian defined by
$$
F_\dd : [C;x_1,\dots,x_d] \mapsto \bigg[\sum_{j=1}^n d_j x_j\bigg] \in \Jac C
$$
when $C$ is smooth. It is holomorphic over $\Mbar_{g,n}-\Delta_0^\sing$, the
complement of the singular locus of $\D_0$.

Denote the class of the zero section of $\J_g$ in $H^{2g}(\J_g,\Q)$ by
$\eta_g$.

\begin{problem}[Eliashberg]
Compute the class in $H^{2g}(\Mbar_{g,n}-\D_0^\sing)$ of the pullback
$F_\dd^\ast \eta_g$ of the zero section of $\J_g$.
\end{problem}

Denote the $j$th Chern class of the Hodge bundle by $\lambda_j$.  When all $d_j$
are zero, the section $F_\dd$ is defined on all of $\Mbar_{g,n}$. Since the
normal bundle of the zero section is $\check\F$, the dual of the Hodge bundle,
we have:

\begin{proposition}
\label{prop:res_lambda}
If $\dd = 0$, then $F_\dd^\ast\eta_g = (-1)^g \lambda_g \in
H^{2g}(\Mbar_{g,n},\Q)$.
\end{proposition}

\begin{remark}
This result also holds in the Chow ring.
\end{remark}

\section{Families of Compact Tori}
\label{tori}

The restriction of the universal jacobian to $\M_{g,n}^c$ is a family of compact
tori. This section is a discussion of some general properties of families of
compact tori.

\begin{definition}
A family of compact (real) $r$-dimensional tori is a smooth fiber bundle $f: T
\to B$, each of whose fibers is a compact, connected abelian Lie group. This
bundle is locally (but typically not globally) trivial as a bundle of Lie
groups.
\end{definition}

We shall assume throughout that $B$ is connected. The identity section will be
denoted by $s : B \to T$. Denote the fiber of $T$ over $b\in B$ by $T_b$.

For a coefficient ring $R$, denote the local system over $B$ whose fiber over
$b$ is $H_1(T_b,R)$ by $\H_R$. The following assertion is easily proved.

\begin{proposition}
If $f : T \to B$ is a family of compact tori, then there is a natural
bijection
$$
T \to \H_\R/\H_\Z
$$
which commutes with the projections to $B$ and is a group homomorphism on each
fiber. \qed
\end{proposition}

The flat structure on $\H_\R$ descends to a flat structure on $T=\H_\R/\H_\Z$.

\begin{corollary}
Every bundle of compact tori has a natural flat structure in which the torsion
multi-sections are leaves. Equivalently, a bundle of compact real tori has a 
natural trivialization over each contractible subset of $B$ in which the torsion
sections are constant. \qed
\end{corollary}

\subsection{The class of a section}
\label{sec:sect_class}

Each section $s$ of a family $T\to B$ of compact tori determines a class $c(s)
\in H^1(B,\H_\Z)$. We review three standard constructions of this class.

The first is sheaf theoretic. Denote the sheaf of $C^\infty$ real-valued
functions on $B$ by $\cE_B$. The flat vector bundle associated to $\H_\R$ has
sheaf of sections $\cH := \H_\R \otimes_\R \cE_B$. Denote the sheaf of smooth
sections of $T\to B$ by $\T$. Then one has a short exact sequence
$$
0 \to \H_\Z \to \cH \to \T \to 0
$$
of sheaves. Taking cohomology yields the exact sequence
$$
0 \to H^0(B,\H_\Z) \to H^0(B,\cH) \to H^0(B,\T) \overset{c}{\to} H^1(B,\H_\Z)
\to 0.
$$
The connecting homomorphism is well defined up to a sign. With the appropriate
choice it takes a section $s$ to its characteristic class. Note that the
vanishing of $c(s)$ implies that $s$ is homotopic to the zero section. Thus
$H^1(B,\H_\Z)$ can be identified with the group of homotopy classes of smooth
sections of $T \to B$.

The second description is obtained by regarding $H^1(B,\H_\Z)$ as congruence
classes of extensions
$$
0 \to \H_\Z \to \E \to \Z_B \to 0
$$
of local systems over $B$. Given a section $s$ of $T\to B$, we can construct
such an extension $\E$ as the local system whose fiber over $b\in B$ is
$H_1(\That_b,\Z)$, where
$$
\That_b := T_b\cup_h [0,1]
$$
where $h(0)= 0$ and $h(1)=s(b)$. There is a short exact sequence
$$
0 \to H_1(T_b) \to H_1(\That_b) \to \Z \to 0
$$
in which $H_1([0,1],\{0,1\})$ is identified with $\Z$ by taking the generator to
be the class of a path from $1$ to $0$.

When $s(b)\neq 0$,
$H_1(\That_b) \cong H_1(T_b,\{0,s(b)\}) \cong H^{r-1}(T_b-\{0,s(b)\})$.
The first description of $c$ is determined only up to a sign. This description
fixes the sign.

The third description uses de Rham cohomology. Each $b\in B$ has an open
neighbourhood $U$ where $s$ lifts to a section $\stilde :  U \to \cH$ of the
flat vector bundle $\cH$. When $U$ is connected, such a lift is unique up to
translation by a local section of $\H_\Z$. The 1-form $d\stilde$ on $U$ with
values in $\H_\R$ is therefore independent of the choice of the lift $\stilde$.
The de Rham representative of $c(s)$ is the class that is locally represented by
$d\tilde{s}$.

\subsubsection{Equivalence of these constructions}
Here is a quick sketch of the equivalence of these definitions. Choose an open
covering $\cU = \{U_\alpha\}$ of $T$ such that $U_{\alpha_0} \cap \dots \cap
U_{\alpha_q}$ is contractible for all multi-indices $(\alpha_0,\dots,\alpha_q)$.
Such an open covering can be constructed by taking each $U_\alpha$ to be a
geodesically convex ball with respect to some riemannian metric on $T$. The
complex $C^\dot(\cU,\F)$ of {C}ech cochains with coefficients in $\F$ computes
$H^\dot(T,\F)$ when $\F$ is $\H_\Z$, $\H_\R$, $\cH$ and $\T$. Suppose that $s$
is a smooth section of $T \to B$. Its restriction to $U_\alpha$ can be lifted to
a section $s_\alpha$ of $\cH$. The difference $c_{\alpha\beta} := s_\beta -
s_\alpha$ is a section of $\H_\Z$ over $U_\alpha\cap U_\beta$. The class of
$c(s)$ is represented by the cocycle $(c_{\alpha\beta}) \in C^1(\cU,\H_\Z)$.

The class of an extension
$$
0 \to \H_\Z \to \E \to \Z_T \to 0.
$$
is computed by choosing sections $e_\alpha$ of $\E \to \Z_T$ over each
$U_\alpha$. The class of the extension is represented by the cocycle
$(e_\beta-e_\alpha)_{\alpha\beta} \in C^1(\cU,\H_\Z)$. In the case above, one
can take the local section $e_\alpha$ to be the class in $H_1(T_b,\{0,s(b)\})$
of the image of the path in the universal covering $\H_{\R,b}$ of $T_b$ that
goes from $0$ to $s_\alpha(b)$. Then $e_\beta - e_\alpha = c_{\alpha\beta}$, as
required.

Denote the de~Rham sheaf of smooth $\R$-valued forms on $T$ by $\cE^\dot_T$.
Standard arguments imply that the inclusions
$$
C^\dot(\cU,\H_\R) \hookrightarrow C^\dot(\cU,\cE^\dot_T\otimes \H_\R)
\hookleftarrow \cE^\dot_T \otimes \H_\R
$$
induce isomorphisms on homology. The standard zig-zag argument implies that
the class of the cocycle $(c_{\alpha\beta})$ is represented by the element of
$$
E^\dot(T,\H_\R) := H^0(T,\cE^\dot_T \otimes \H_\R)
$$
whose restriction to $U_\alpha$ is $ds_\alpha$.

\subsection{Invariant cohomology classes}

The flat structure of a family of compact tori $T\to B$ can be used to construct
natural de Rham representatives of cohomology classes on $T$. We will say that a
differential form $w$ on a manifold $M$ with a foliation $\L$ is {\it parallel
with respect to} $\L$ if the Lie derivative of $w$ with respect to each vector
field tangent to $\L$ vanishes. A family of tori is foliated as it is a flat
family of tori.

The following lemma is proved in \cite[Lemma~5.1]{hain-reed:arakelov}.

\begin{lemma}
\label{lem:gen_lifts}
If $f : T \to B$ is a family of compact tori, there is a natural mapping
$$
\sigma : H^0(B, R^k f_\ast \R) \to H^k(T,\R)
$$
whose composition with the projection
$$
H^k(T,\R) \to H^0(B, R^k f_\ast \R)
$$
is the identity. Moreover, for each $u\in H^0(B,R^k f_\ast \R)$, the extended
class $\sigma(u)$ has a natural differential form representative $w_u$ whose
restriction to each fiber is a closed, translation-invariant differential form,
and which is parallel with respect to the flat structure. This class has the
property that its restriction to every leaf (such as the zero section and every
torsion multi-section) is zero. \qed
\end{lemma}

\subsection{The Poincar\'e dual of the zero section}

Let $r$ be the real dimension of the fiber of $f : T \to B$. If $B$ and $T$ are
oriented and $B$ is connected, then
$$
H^0(B,R^r f_\ast \R) \cong \R.
$$
Let $u$ be the element of this group whose value on one (and hence all) fibers
is $1$. Denote the class $\sigma(u) \in H^r(T,\R)$ by $\psi$.

\begin{proposition}
\label{cpact_tori}
If the base $B$ is a compact manifold (possibly with boundary), then the
Poincar\'e dual of the zero section is $\psi$.
\end{proposition}

\begin{proof}
Set $d = \dim_\R B$. For $e\in \Z$ define $[e] : T \to T$ to be the map whose
restriction to each fiber is multiplication by $e$. Since $[e]$ induces
multiplication by $e^k$ on $R^k f_\ast \Q$, it follows that the Leray spectral
sequence degenerates at $E_2$. It also follows that the eigenvalues of the
induced mapping $[e]^\ast$ on $H^k(T)$ lie in $\{1,e,\dots,e^k\}$. Since none
of these eigenvalues is zero when $e\neq 0$, $[e]^\ast$ is invertible on
rational cohomology when $e\neq 0$. Similar assertions hold for the homology
spectral sequence
$$
H_s(B,\partial B;\H_t) \implies H_{s+t}(T,\partial T),
$$
where $\H_t$ denotes the local system whose fiber over $b\in B$ is $H_t(T_b)$.
Since
$$
H_{r+d}(T,\partial T) = H_d(B,\partial B; \H_r)
$$
it follows that $[e]_\ast [T] = e^r[T]$, where $[T]$ denotes the fundamental
class of $(T,\partial T)$.

Now assume that $e>1$. Then the collapsing of the Leray spectral sequence
implies that the dimension of the $e^r$-eigenspace of $H^r(T)$ is one.
Since the form $w_\psi$ that naturally represents $\psi$ has the property that
$$
[e]^\ast w_\psi = e^r w_\psi,
$$ 
and since $\psi$ is non-trivial (as it has non trivial integral over a fiber),
it follows that $\psi$ spans this eigenspace.

Note that the class $[Z] \in H_d(T,\partial T)$ of the zero-section is an
eigenvector of $[e]_\ast$ with eigenvalue $1$. Denote the Poincar\'e dual of the
zero section by $\eta_Z$. It is characterized by the property that
$$
[T] \cap \eta_Z = [Z] \in H_d(T,\partial T),
$$
where $\cap$ denotes the cap product \cite[p.~254]{spanier}
$$
\cap : H_{d+r}(T,\partial T) \otimes H^r(T) \to H_d(T,\partial T).
$$
Since $e_\ast[Z]=[Z]$, standard properties of the cap product \cite[Assertion
16, p.~254]{spanier}, we have
$$
[e]_\ast\big([T] \cap [e]^\ast \eta_Z\big) = \big([e]_\ast [T]\big) \cap \eta_Z
= e^r [T]  \cap \eta_Z
= [e]_\ast \big([T]  \cap e^r \eta_Z \big).
$$
Since $[e]_\ast$ and capping with $[T]$ are both isomorphisms, it follows that
$[e]^\ast \eta_Z = e^r \eta_Z$. Since $\eta_Z$ and $\psi$ both lie in the
$e^r$-eigenspace and agree on each fiber, they are equal.
\end{proof}

Since the normal bundle of the zero section $Z$ is the flat bundle $\H_\R$,
the Euler class of the normal bundle of $Z$ vanishes in rational cohomology.

\begin{corollary}
\label{vanishing}
The restriction of the Poincar\'e dual of the zero section $Z$ of a family of
compact tori to $Z$ vanishes in rational cohomology. \qed
\end{corollary}

Combined with Proposition~\ref{prop:res_lambda}, this implies the well-known
property of the top Chern class of the Hodge bundle.

\begin{corollary}
The restriction of $\lambda_g$ to $\M_g^c$ vanishes in rational cohomology.
\end{corollary}

\subsection{The class $S\circ c(s)^2$}

Suppose that $f:T\to B$ is a flat family of tori and that $\H$ is the
corresponding local system such that $T=\H_\R/\H_\Z$. A flat, skew-symmetric
inner product $S : \H_\R \otimes \H_\R \to \R$ gives an element of
$H^0(T,R^2f_\ast\R_T)$. Lemma~\ref{lem:gen_lifts} implies that $S$ determines a
closed $2$-form $\phi_S$ on $T$. It is characterized by the properties:
\begin{enumerate}

\item its restriction to the fiber $T_b$ of $T$ is the translation invariant
$2$-form on $T_b$ that corresponds to $S$,

\item it is parallel with respect to the flat structure on $T$,

\item its restriction to the zero-section of $T$ is zero.

\end{enumerate}

The following result is easily proved using the de~Rham description of $c(s)$.

\begin{proposition}
If $s$ is a holomorphic section of $T\to B$, then the form $s^\ast \phi_S$
represents the cohomology class $S\circ c(s)^2 \in H^2(T,\R)$. \qed
\end{proposition}

\section{Normal Functions}
\label{sec:normal}

Normal functions are our primary tool. They are holomorphic sections of families
of intermediate jacobians that satisfy certain infinitesimal and asymptotic
conditions. In this section, we recall Griffiths construction of intermediate
jacobians and of the normal function associated to a family of homologically
trivial algebraic cycles in a family of smooth projective varieties.

\subsection{Intermediate Jacobians}

Suppose that $Y$ is a compact K\"ahler manifold and that $Z$ is an
algebraic $d$-cycle in $Y$ where $0\le d < \dim Y$. One has the
exact sequence
$$
0 \to H_{2d+1}(Y) \to H_{2d+1}(Y,Z) \to H_{2d}(|Z|) \to H_{2d}(Y) \to \cdots
$$
of integral homology groups associated to the pair $(Y,|Z|)$, where $|Z|$
denotes the support of $Z$. It is an exact sequence of mixed Hodge structure.
The class of the cycle $Z$ defines a morphism of mixed Hodge structures
$$
c_Z : \Z(d) \to H_{2d}(|Z|),
$$
where $\Z(d)$ denotes the Hodge structure of type $(-d,-d)$ whose underlying
lattice is isomorphic to $\Z$. If $Z$ is null homologous, we can pull back the
above sequence along $c_Z$ to obtain an extension
$$
0 \to H_{2d+1}(Y) \to E_Z \to \Z(d) \to 0
$$
in $\hodge$, the category of mixed Hodge structures. Tensoring with $\Z(-d)$
gives an extension
$$
0 \to H_{2d+1}(Y,\Z(-d)) \to E_Z(-d) \to \Z(0) \to 0
$$
and thus a class $e_Z$ in
$$
\Ext^1_\hodge(\Z(0),H_{2d+1}(Y,\Z(-d))).
$$
Note that, since $H_{2d+1}(Y)$ has weight $-(2d+1)$, $H_{2d+1}(Y,\Z(-d))$ has
weight $-1$.

Suppose that $V$ is a Hodge structure of weight $-1$ whose underlying lattice
$V_\Z$ is torsion free. The associated jacobian
$$
J(V) := V_\C/(V_\Z + F^0 V_\C)
$$
is a compact complex torus. In general, $J(V)$ is not an abelian variety. When
$V$ is the weight $-1$ Hodge structure $H_{2d+1}(Y,\Z(-d))$ (mod its torsion),
$J(V)$ is the $d$th Griffiths intermediate jacobian of $Y$.

There is a natural isomorphism (see \cite{carlson} or \cite[\S
3.5]{peters-steenbrink}, for example)
$$
\Ext^1_\hodge(\Z(0),V)  \cong J(V).
$$
The class $e_Z$ of a homologically trivial $d$-cycle $Z$ in $Y$ can thus be
viewed as a class
$$
e_Z \in J(H_{2d+1}(Y,\Z(d)))
$$
in the $d$th Griffiths intermediate jacobian. This class can be described
explicitly by Griffiths' generalization \cite{griffiths} of the Abel-Jacobi
construction , which we now recall. First observe that the standard pairing
between $H_{2d+1}(Y)$ and $H^{2d+1}(Y)$ induces an isomorphism
$$
H_{2d+1}(Y,\Z(-d))/F^0 \cong \Hom_\C(F^{d+1}H^{2d+1}(Y),\C).
$$
The natural mapping $H_{2d+1}(Y,\Z(-d)) \to H_{2d+1}(Y,\Z(-d))/F^0$ corresponds
to the integration mapping
$$
H_{2d+1}(Y,\Z) \to \Hom_\C(F^{d+1}H^{2d+1}(Y),\C)
$$
that takes the homology class $z$ to $\xi \mapsto \int_z \xi$. A homologically
trivial $d$-cycle $Z$ in $Y$ can be written as the boundary $\partial \G$ of a
(topological) $(2d+1)$-chain $\G$. Note that $\G$ is well defined up to the
addition of an integral $(2d+1)$-cycle. Classical Hodge theory implies that each
element $u$ of $F^{d+1}H^{2d+1}(Y)$ can be represented by a closed $C^\infty$
form $\xi_u$ in the $(d+1)$st level of the Hodge filtration on the de~Rham
complex of $Y$ and that any two such forms differ by the exterior derivative of
a form in the same level $F^{d+1}$ of the de~Rham complex. The point $e_Z$ is
represented by the element
$$
\int_\G : u \mapsto \int_\G \xi_u.
$$
of 
$$
\Hom_\C(F^{d+1}H^{2d+1}(Y),\C)/H_{2d+1}(Y,\Z) \cong J(H_{2d+1}(Y,\Z(-d))).
$$
Stokes' Theorem implies that the image of this functional in the intermediate
jacobian depends only on $Z$ and not on the choice of $\G$ or $\xi_u$.

This construction generalizes the classical construction for $0$-cycles on
curves that was sketched in Section~\ref{sec:eliashberg}. More generally, it
generalizes the classical construction for $0$-cycles, where $J(H_1(Y)) \cong
\Alb Y$, and for divisors, where $J(H_{2d-1}(Y))$ $\cong \Pic^0 Y$ and
$d=\dim Y$.

\subsection{Normal Functions}

Suppose that $\Xbar$ is a complex projective manifold and that $X = \Xbar - D$
where $D$ is a normal crossings divisor in $\Xbar$. Suppose that $\V$ is a
variation of Hodge structure over $X$ of weight $-1$. Denote by $J(\V)\to X$ the
corresponding bundle of intermediate jacobians; the fiber over $x\in X$ is
$J(V_x)$, where $V_x$ is the fiber of $\V$ over $x$. It is a family of compact
tori.

The discussion of Section~\ref{sec:sect_class} implies that a holomorphic
section $\nu : X \to J(\V)$ determines a cohomology class $c(\nu) \in H^1(X,\V)$
and a local system $\E \to X$ which is an extension
$$
0 \to \V \to \E \to \Z_X \to 0.
$$
The point $\nu(x) \in J(V_x)\cong \Ext^1_\hodge(\Z(0),V_x)$ determines a mixed
Hodge structure on the fiber $E_x$ of $\E$ over $x\in X$ so that
$$
0 \to V_x \to E_x \to \Z(0) \to 0
$$
is an extension in $\hodge$. That $\nu$ is holomorphic implies that this family
of MHSs varies holomorphically with $x\in X$.

\begin{example}
\label{ex:cycles}
Families of homologically trivial algebraic cycles give rise to such extensions.
Suppose that $\Ybar$ is a complex projective manifold and that $f:\Ybar \to
\Xbar$ is a morphism whose restriction to $X$ is a family $Y \to X$ of
projective manifolds. Suppose that that $Z$ is an algebraic $d$-cycle in $Y$
such that the restriction $Z_x$ of $Z$ to the fiber over each $x\in X$ is a
homologically trivial $d$-cycle in $Y_x$. Applying the construction of the
previous section fiber-by-fiber, one obtains an extension $\E_Z$ of $\Z_X(0)$ by
the variation of Hodge structure $\V$ of weight $-1$ whose fiber over $x \in X$
is $V_x = H_{2d+1}(Y_x,\Z(-d))$. The family $\{E_x\}_{x\in X}$ of extensions of
MHS corresponds to a holomorphic section $\nu$ of the bundle $J(\V) \to X$ of
intermediate jacobians.
\end{example}

The section $\nu$ is not an arbitrary holomorphic section. It satisfies the {\em
Griffiths infinitesimal period relation}\footnote{This states that if
$\nutilde$ is a local holomorphic lift of a normal function $\nu$ to a section
of $\E\otimes\O_X$, then its derivative $\nabla \nutilde$ is a section of
$\F^{-1}(\E\otimes\O_X)\otimes \Omega^1_X$.} at each $x\in X$ and also satisfies
strong (and technical) conditions as $Y_x$ degenerates. A succinct way to state
these conditions is to say that the corresponding family $\E$ of MHS is an {\em
admissible variation of MHS} in the sense Steenbrink-Zucker
\cite{steenbrink-zucker} and Kashiwara \cite{kashiwara}. (The definition and an
exposition of admissible variations of MHS can be found in \cite[\S
14.4.1]{peters-steenbrink}.)

All ``naturally defined local systems'' over a smooth variety $X$ that arise
from families of varieties over $X$, such as the one constructed in
Example~\ref{ex:cycles}, are admissible variations of MHS.\footnote{These
results are due to Steenbrink-Zucker \cite{steenbrink-zucker}, Navarro et al
\cite{guillen} and Saito \cite {saito:mhm}. Precise statements can be found in
\cite[Thm.~14.51]{peters-steenbrink}.} The admissibility conditions axiomatize
the infinitesimal and asymptotic properties satisfied by such geometrically
defined local systems.


\begin{definition}[{Hain \cite[\S6]{hain:normal}, Saito \cite{saito}}]
\label{def:normal}
A section $\nu$ of a family $J(\V)\to X$ of intermediate jacobians is a {\it
normal function} when the corresponding family of MHS $\E$ is an admissible
variation of MHS.
\end{definition}

The preceding discussion implies that the section $\nu$ of $J(\V) \to X$
associated to a family of null homologous cycles is a normal function. Several
concrete examples of normal functions over moduli spaces of curves will be given
in Section~\ref{sec:normal_geom}. A detailed discussion of normal functions can
be found in \cite[\S 2.11]{kerr-pearlstein}, where they are called {\em
admissible} normal functions.

Denote the category of admissible variations of mixed Hodge structure over a
smooth variety $X$ by $\hodge(X)$. It is abelian. The definition of normal
functions implies that normal function sections of $J(\V) \to X$ correspond to
elements of $\Ext^1_{\hodge(X)}(\Z_X(0),\V)$.  In the appendix to
\cite{hain:rat_pts} it is proved that one has an exact sequence
\begin{equation}
\label{eqn:seqce}
0 \to \Ext^1_\hodge(\Z(0),H^0(X,\V_\Z)) \overset{j}{\to}
\Ext^1_{\hodge(X)}(\Z_X(0),\V_\Z)
\overset{\delta}{\to} H^1(X,\V_\Z)
\end{equation}
where $\delta$ takes a normal function $\nu$ to its class $c(\nu)$. An immediate
consequence is the following rigidity property of normal functions.

\begin{proposition}
\label{prop:class}
If $H^0(X,\V_\Q)=0$, then the group $\Ext^1_{\hodge(X)}(\Z(0)_X,\V_\Z)$ is
finitely generated and each normal function section $\nu$ of $J(\V)$ is
determined, up to a torsion section, by its class $c(\nu) \in H^1(X,\V)$.
\end{proposition}

\subsection{Extending normal functions}

The following result guarantees that the normal functions that we will define
over $\M_{g,n}$ extend to $\M_{g,n}^c$. A proof can be found, for example, in
\cite[Thm.~7.1]{hain:normal}.

\begin{proposition}
\label{prop:extends}
Suppose that $\V$ is a variation of Hodge structure of weight $-1$ over a smooth
variety $X$. If $\nu$ is a normal function section of $J(\V)$ that is defined on
the complement $X-Y$ of a closed subvariety $Y$ of $X$, where $Y\neq X$, then
$\nu$ extends across $Y$ to a normal function section of $J(\V)$ over $X$.
\end{proposition}

Note that, in this result, the variation $\V$ is defined over $X$, not just over
$X-Y$. The proposition asserts that normal functions defined generically on $X$
extend to $X$. It does not assert that if $\V$ is defined only over $X-Y$, then
a normal function section of $J(\V)$ over $X-Y$ will extend to $X$. Before
discussing this problem, one first has to construct an extension of $J(\V)$ to
$X$. The existence of such extensions is discussed in \cite{zucker} and
\cite{kerr-pearlstein}, for example.

\subsection{Some variations of Hodge structure}
\label{sec:vhs}

Denote the moduli stack of principally polarized abelian varieties of dimension
$g$ by $\A_g$.  Here we introduce three variations of Hodge structure over
$\A_g$ whose pullbacks to $\M_{g,n}^c$ have a special geometric significance.
Throughout we suppose that $g\ge 2$.

Denote the universal abelian variety over $\A_g$ by $f:\X \to \A_g$.
The local system
$$
\H := R^1 f_\ast \Z_\X(1)
$$
is a variation of Hodge structure over $\A_g$ of weight $-1$. Its pullback to
$\M_{g,n}^c$ along the period mapping $\M_{g,n}^c \to \A_g$ will also be denoted
by $\H$. The corresponding family of intermediate jacobians $J(\H)$ over $\A_g$
is naturally isomorphic to $\X$, the universal principally polarized abelian
variety.

The construction of the universal jacobian $\J_g$ in
Section~\ref{sec:eliashberg} implies that its restriction to $\M_g^c$ is a
bundle of intermediate jacobians. 

\begin{proposition}
\label{prop:jac}
If $g\ge 2$, then $J(\H) \to \A_g$ is naturally isomorphic to the universal
abelian variety and its pullback to $\M_g^c$ is isomorphic to the universal
jacobian $\J_g^c$.
\end{proposition}

The canonical polarization defines a morphism $S_H : \H\otimes\H \to \Z(1)$ into
the constant variation of Hodge structure $\Z(1)$. It can be regarded as a
section of $(\Lambda^2\H)(-1)$, which we shall also denote by $S_H$.

Denote by $\bL$ the variation of Hodge structure $(\Lambda^3\H)(-1)$. It has
weight $-1$. Wedging with $S_H$ defines an inclusion
$$
\H \hookrightarrow \bL,\qquad x\mapsto x\wedge S_H
$$
of variations of Hodge structure. Set $\V = \bL/\H$. Note that when $g=2$,
$\bL\cong \H$ and $\V=0$. Denote the fibers of $\H$, $\bL$ and $\V$ over the
moduli point of $\Jac C$ by $H_C$, $L_C$ and $V_C$, respectively.

\begin{remark}
The variation $\bL$ over $\M_g^c$ is isomorphic to the variation $R^3 \pi_\ast
\Z(2)$ where $\pi : \J_g^c \to \M_g$ denotes the projection. Its fiber over
$[C]$ is $H^3(\Jac C,\Z(2))$. The twist $\Z(2)$ lowers the weight from $3$ to
$-1$. Its quotient $\V$ is an integral form of the $\Q$-variation of Hodge
structure whose fiber over $[C]$ is the primitive part of $H^3(\Jac C,\Q(2))$.
\end{remark}

\subsection{Normal Functions over $\M^c_g$ and $\M_{g,n}^c$}
\label{sec:normal_geom}

Suppose that $g\ge 1$. An irreducible representation $V$ of $\Sp_g$ determines a
variation of Hodge structure $\V$ over $\M_{g,n}$ that is unique up to Tate
twist.\footnote{This is very well-known. A proof can be found, for example, in
\cite[\S9]{hain:normal}.} Since every such variation of Hodge structure $\V$
extends to $\M_{g,n}^c$, Proposition~\ref{prop:extends} implies that every
normal function section of $J(\V)$ defined over $\M_{g,n}$ extends to a normal
function on $\M_{g,n}^c$.

By Proposition~\ref{prop:jac}, the bundle of intermediate jacobians that
corresponds to the fundamental representation of $\Sp_g$ is the pullback of the
universal jacobian $J(\H) = \J_g$ to $\M_{g,n}^c$. Its group of sections
$\Ext^1_{\hodge(\M_{g,1})}(\Z(0),\H)$ is free of rank 1. Its generator $\K$ is
the normal function that takes the moduli point $[C,x]$ of the pointed curve
$(C,x)$ to
$$
\K([C,x]) := (2g-2)[x] - K_C \in \Jac(C)
$$
where $K_C$ denotes the canonical divisor class of $C$.

When $n\ge 1$, we can pull back $\K$ along the $j$th projection $\M_{g,n}
\to \M_{g,1}$ to obtain the normal function
$$
\K_j : \M_{g,n} \to J(\H)\qquad j = 1,\dots, n.
$$
Explicitly
$$
\K_j([C,x_1,\dots,x_n]) = (2g-2)[x_j] - K_C \in \Jac(C).
$$

When $n\ge 2$ we also have the normal functions
$$
\cD_{j,k} : \M_{g,n} \to J(\H)\qquad 1 \le j<k \le n
$$
defined by
$$
\cD_{j,k}([C,x_1,\dots,x_n]) = [x_j] - [x_k] \in \Jac(C).
$$

\begin{proposition}[{cf.\ \cite[Thm.~12.3]{hain:normal}}]
If $g\ge 3$ and $n\ge 2$, then the group of normal function sections (indeed,
all sections) of $J(\H) \to \M_{g,n}$ is torsion free and is generated by
$\K_1,\dots,\K_n$ and the $\cD_{j,k}$ where $1\le j < k \le n$.
\end{proposition}

The most interesting normal function over $\M_{g,n}$ is constructed from the
Ceresa cycle in the universal jacobian. Suppose that $C$ is a smooth projective
curve of genus $g\ge 3$ and that $x\in C$. Then one has the imbedding
$$
\mu_x : C \to \Jac C
$$
that takes $y$ to $[y]-[x]$. Its image is an algebraic $1$-cycle in $\Jac C$
that we denote by $C_x$. Let $i$ be the involution $u\mapsto -u$ of $\Jac C$.
Set $C_x^-=i_\ast C_x$. Since $i^\ast : H^1(\Jac C) \to H^1(\Jac C)$ is
multiplication by $-1$, it follows that $i^\ast : H^k(\Jac C) \to H^k(\Jac C)$
is multiplication by $(-1)^k$. This implies that the $1$-cycle $C_x - C_x^-$,
called the {\em Ceresa cycle}, is homologically trivial. It therefore determines
a class
$$
\nu_x(C) \in J(H_3(\Jac C,\Z(-1)))
$$
and a normal function:
$$
\xymatrix{
J(\bL) \ar[r] & \M_{g,1}\ar@/_1pc/[l]|\nutilde
}
$$
whose value at $[C,x]$ is $\nu_x(C)$. The inclusion $\H \hookrightarrow \bL$
induces an inclusion
$$
j: \J_{g,1} = J(\H) \hookrightarrow J(\bL).
$$
It is proved in \cite{pulte} that if $x,y\in C$, then
$$
\nutilde_x(C) - \nutilde_y(C) = 2j([x]-[y]) \in J(L_C)
$$
so that the image $\nu(C)$ of $\nu_x(C)$ in $J(V_C)$ does not depend on the
choice of $x\in C$. This implies that $\nutilde$ is pulled back from a normal
function
$$
\xymatrix{
J(\V) \ar[r] & \M_{g}\ar@/_1pc/[l]|\nu.
}
$$
We will abuse notation and also denote its pullback to $\M_{g,n}$ by $\nu$.

\begin{proposition}[{\cite[Thm.~8.3]{hain:normal}}]
If $g\ge 3$ and $n\ge 0$, then the group of normal function sections of $J(\V)
\to \M_{g,n}$ is freely generated by $\nu$.
\end{proposition}

Proposition~\ref{prop:extends} implies that the normal functions $\nu, \K_j,
\delta_{j,k}$ extend canonically to normal functions over $\M_{g,1}^c$.

\section{Biextension Line Bundles}
\label{sec:biextensions}

This section is a brief review of facts about biextension line bundles from
\cite{hain:biext}, \cite{lear}, and \cite{hain-reed:arakelov}. It is needed to
prove that the square of the class of a normal function extends naturally to a
class on $\Mbar_{g,n}$ even though the normal function itself does not extend.

Suppose that $\U$ is a variation of Hodge structure over an algebraic manifold
$X$ of weight $-1$ endowed with a flat inner product $S$ that satisfies the
condition
$$
S\in \Hom_\hodge(U_x^{\otimes 2},\Z(1))\text{ for all }x\in X.
$$
Equivalently,
$$
S(U_x^{p,q},\overline{U_x^{r,s}}) = 0 \text{ unless $p=r$ and $q=s$}.
$$
Set $\Udual := \Hom_\Z(\U,\Z_X(1))$. This also a variation of Hodge structure of
weight $-1$. There is a natural isomorphism
$$
\Ext^1_{\hodge(X)}(\Z_X(0),\Udual) \cong \Ext^1_{\hodge(X)}(\U,\Z_X(1)).
$$

The {\em biextension line bundle $\B$} is a line bundle over $J(\U)\times_X
J(\Udual)$. Denote the associated $\Gm$-bundle by $\B^\ast$. The fiber of the
projection
\begin{equation}
\label{eqn:biext}
\B^\ast \to J(\U)\times_X J(\Udual) \to X
\end{equation}
over $x\in X$ is the set of all mixed Hodge structures whose weight graded
quotients are $\Z(0)$, $V_x$ and $\Z(1)$ via a fixed isomorphism. These are
called {\em biextensions}. The projection (\ref{eqn:biext}) takes the
biextension $B$ to the pair of extensions $B/\Z(1)$ and $W_{-1}B$. A detailed
exposition of the construction of $\B$ is given in
\cite[\S7]{hain-reed:arakelov}.

A (Hodge) biextension is a section $\beta$ of (\ref{eqn:biext}) that corresponds
to an admissible variation of MHS $\bB$ over $X$ with weight graded quotients
$\Z_X(0)$, $\U$ and $\Z_X(1)$. Its fiber over $x\in X$ is the biextension
$\beta(x)$. The composite of a biextension $\beta$ with the projection $\B^\ast
\to J(\U)\times_X J(\Udual)$ is a pair of normal functions that determines the
extension $\bB/\Z(1)$ of $\Z_X(0)$ by $\U$ and the extension $W_{-1}\bB$ of $\U$
by $\Z_X(1)$. The biextension line bundle has a canonical metric $|\blank|_\B$.
A biextension $\beta$ thus determines the real-valued function $\log|\beta|_\B :
X \to \R$.

The pairing $S$ also defines a morphism $\U \to \Udual$ of variations of Hodge
structure over $X$, and therefore a map $i_S : J(\U)\to J(\Udual)$. Pulling back
the line bundle $\B$ along the map
$$
(\id,i_S) : \J(\U) \to J(\U)\times_X J(\Udual)
$$
we obtain a metrized line bundle $\Bhat \to J(\U)$. By
\cite[Prop.~7.3]{hain-reed:arakelov}, the curvature of $\Bhat$ is the
translation-invariant, parallel 2-form $2\w_S$ on $J(\U)$ that corresponds to
the bilinear form $2S$. Points of the associated $\C^\ast$-bundle $\Bhat^\ast$
correspond to ``symmetric biextensions''.

Denote by $\phi_S \in H^2(J(\U))$ the class of $\w_S$. Since $2\w_S$ represents
$c_1(\Bhat)$, the class $2\phi_S$ is integral.

Suppose now that $X = \Xbar - Y$, where $\Xbar$ is smooth and $Y$ is a
subvariety. Each normal function section $\nu$ of $J(\U)$ thus determines a
metrized holomorphic line bundle $\nu^\ast \Bhat$ over $X$. One can ask whether
it extends as a metrized line bundle to $\Xbar$. Lear's thesis \cite{lear}
implies that a power of it extends to a {\em continuously} metrized holomorphic
line bundle over $X - Y^\sing$.

\begin{theorem}[Lear \cite{lear}]
\label{thm:lear}
If $\dim X = 1$ and $\nu$ is a normal function section of $J(\U) \to X$, then
there exists an integer $N\ge 1$ such that the metrized holomorphic line bundle
$\nu^\ast\Bhat^{\otimes N}$ over $X$ extends to a holomorphic line bundle over
$\Xbar$ with a continuous metric. Moreover, if $\beta$ is a biextension section
defined over $X$ that projects to $\nu$, and if $\bD$ is a disk in $\Xbar$ with
holomorphic coordinate $t$ that is centered at a point of $\Xbar-X$, then there
is a rational number $p/q$, which depends only on the monodromy about the origin
of $\bD$ of the VMHS over $\bD^\ast$ corresponding to $\beta$, such that
\begin{equation}
\label{eqn:asymptotics}
\big|\log |\beta(t)|_\B - \frac{p}{q}\log|t|\big|
\end{equation}
is bounded in a neighbourhood of $t=0$.
\end{theorem}

Note that the continuity of the metric ensures that the extension is uniquely
determined. Since $\Xbar$ is smooth, every line bundle over $\Xbar-Y^\sing$
extends uniquely to a line bundle over $\Xbar$. Lear's Theorem thus implies the
following result in the case $\dim X \ge 1$.

\begin{corollary}
If $\nu$ is a normal function section of $J(\U) \to X$, then there exists an
integer $N\ge 1$ such that the metrized holomorphic line bundle
$\nu^\ast\Bhat^{\otimes N}$ over $X$ extends to a holomorphic line bundle
$\Bbar_{N,\nu}$ over $\Xbar$ whose metric extends continuously over
$\Xbar-Y^\sing$.
\end{corollary}

This result implies that the class $\nu^\ast\phi_S \in H^2(X)$ of $\nu^\ast\w_S$
has a natural extension to a class in $H^2(\Xbar)$; namely
$c_1(\Bbar_{N,\nu})/2N$.

\begin{corollary}
\label{cor:extension}
If $\nu$ is a normal function section of $J(\U) \to X$, then the class
$\nu^\ast \phi_S$ has a natural extension to a class
$\widehat{\nu^\ast \phi_S} \in H^2(\Xbar)$.
\end{corollary}

Lear's Theorem implies that the multiplicity of each boundary divisor in
$\widehat{\nu^\ast\phi_S}$ is determined by the asymptotics
(\ref{eqn:asymptotics}) of the restriction of the biextension to a disk
transverse to the boundary divisor.

The previous result suggests that $\nu^\ast\w_S$, regarded as a current on
$\Xbar$, is a natural representative of $\widehat{\nu^\ast\phi_S}$.

\begin{conjecture}
\label{conj:integrable}
If $X$ is a curve, then the $2$-form $\nu^\ast \w_S$ is integrable on $X$ and
$$
\int_X \nu^\ast \w_S = \frac{1}{2N}\int_{\Xbar} c_1(\Bbar_{N,\nu}).
$$
\end{conjecture}

It is known that, in general, the metric does not extend continuously over
$Y^\sing$ due to the phenomenon of ``height jumping'' which we shall discuss in
Section~\ref{sec:moriwaki} and which has been explained by Brosnan and
Pearlstein in \cite{brosnan-pearlstein:heights}.

\section{Polarizations}

Polarizations play an important and subtle (if sometimes neglected) role in
Hodge theory due to their positivity properties.

\subsection{Polarizations}

A {\em polarization} on a Hodge structure $H$ of weight $k$ is a $(-1)^k$
symmetric bilinear form $S$ on $H_\Q$ satisfying the Riemann-Hodge bilinear
relations:
\begin{enumerate}
\item $S(H^{p,q},\overline{H^{r,s}}) = 0$ unless $p=r$ and $q=s$;
\item $i^{p-q}S(v,\vbar) > 0$ when $v\in H^{p,q}$ and $v\neq 0$.
\end{enumerate}
A bilinear form $S$ is a {\em weak polarization} on $H$ if it satisfies the
first condition and the weaker version $i^{p-q}S(v,\vbar) \ge 0$ for all $v\in
H^{p,q}$ of the second.

Suppose that $Y$ is a smooth projective variety of dimension $n$. Denote
the hyperplane class by $w$. For $k\le n$, define a bilinear form $S$
on $H^k(Y)$ by
\begin{equation}
\label{eqn:polarization}
S(u,v) = \int_Y u\wedge v \wedge w^{n-k}.
\end{equation}
This is a non-degenerate, $(-1)^k$ symmetric bilinear form. However, it is
not a polarization in general. The Riemann-Hodge bilinear relations imply
that the restriction of $(-1)^{k(k-1)/2} S$ to $PH^k(Y)$, the primitive
part of $H^k(Y)$, is a polarization. These provide the principal examples of
polarized Hodge structures.

A variation of Hodge structure $\V$ over a base $X$ is {\it polarized} by $S$
if $S$ is a flat bilinear form on the variation which restricts to a
polarization on each fiber.

\subsection{Some polarized variations of Hodge structure over $\A_g$}

The variations $\H$, $\bL$ and $\V$ defined in Section~\ref{sec:vhs} have
natural polarizations. The Riemann bilinear relations imply that the variation
$\H$ over $\A_g$ is polarized by the inner product $S_H$ introduced in
Section~\ref{sec:vhs}. The corresponding polarization is easily described on the
pullback of $\H$ to $\M_g$. In this case, the fiber $H_C$ of $\H$ over the
moduli point $[C]$ of a smooth projective curve $C$ is $H^1(C,\Z(1))$. Under
this isomorphism, $S_H$ corresponds to the inner product
$$
S(u,v) = \int_C u \wedge v.
$$
on $H^1(C)$.

The intersection form $S_H$ extends to the skew symmetric bilinear form
$$
S_L : \bL \otimes \bL \to \Z(1)
$$
defined by
$$
S_L(x_1\wedge x_2 \wedge x_3,y_1\wedge y_2 \wedge y_3) = \det(S_H(x_i,y_j)).
$$
Note that this is {\em not} dual to the inner product $S$ on $H^3(\Jac C)(1)$
defined in equation (\ref{eqn:polarization}) above as is easily seen by a direct
computation.

Denote the fiber $\Lambda^3 H_C$ of $\bL$ over $[C]$ by $L_C$ and the fiber of
$\V$ over $[C]$ by $V_C$. Define $c: L_C \to H_C$ by
\begin{equation}
\label{eqn:contraction}
c(x\wedge y \wedge z) = S_H(y,z)x + S_H(z,x)y + S_H(x,y)z.
\end{equation}
Regard $S_H$ as an element of $\Lambda^2 H_L$. The projection $p: L_C \to V_C$
has a canonical $\Sp_g$-invariant splitting $j$.  It is defined by
$$
j(p(x\wedge y \wedge z)) = x\wedge y \wedge z -
S_H \wedge c(x\wedge y \wedge z)/(g-1).
$$
A skew symmetric bilinear form on $\V$ can be defined by
$$
S_{V}(u,v) = (g-1) S_L(j(u),j(v)).
$$
This form is integral and primitive.

\begin{proposition}
\label{polar}
The variations $(\H,S_H)$, $(\bL,S_L)$ and $(\V,S_{V})$ are polarized
variations of Hodge structure over $\A_g$, as are their pullbacks to
$\M_{g,n}^c$.
\end{proposition}

\begin{proof}
We have already seen that $(\H,S_H)$ is a polarized variation of Hodge
structure. For the rest, it suffices to show that $\Lambda^3 H_1(C)$ is
polarized by $S_L$. To do this, choose a basis $u_1,\dots,u_g$ of $H^{-1,0}$
that is orthonormal under the positive definite hermitian inner product
$$
(u,v) = i^{-1}S_H(u,\vbar).
$$
Then, for example,
$$
i^{-3-0}S_L
(u_1\wedge u_2 \wedge u_3, \ubar_1\wedge \ubar_2,\ubar_3)
= \det(i^{-1}S_H(u_j,\ubar_k)) = 1 > 0.
$$
and
$$
i^{-2-(-1)}S_L
(u_1\wedge u_2 \wedge \ubar_3, \ubar_1\wedge \ubar_2,u_3)
= - i^{-1}i^3\det(i^{-1}S_H(u_j,\ubar_k)) = -i^2 = 1 >0 .
$$
The remaining computations follow by taking complex conjugates.
\end{proof}

The contraction (\ref{eqn:contraction}) induces a projection $c : \bL \to \H$ of
variations of Hodge structure. The canonical quotient mapping $p : \bL \to \V$
is also a morphism of variation of Hodge structure. The polarizations $S_H$ of
$\H$ and $S_V$ of $\V$ can be pulled back along these projections to obtain the
invariant inner product $c^\ast S_H + p^\ast S_V$ on $\bL$. For later use, we
record the following fact:

\begin{lemma}[{\cite[Prop.~18]{hain-reed:chern}}]
\label{lem:reln}
If $g\ge 2$, then $c^\ast S_H + p^\ast S_V = (g-1)S_L$. \qed
\end{lemma}

\section{Cohomology Classes}

By Lemma~\ref{lem:gen_lifts} each invariant inner product on a variation of MHS
$\U$ over a smooth variety $T$ gives rise to a parallel, translation invariant
$2$-form $\w$ on the associated bundle of intermediate jacobians $J(\U)$ and its
cohomology class $\phi \in H^2(J(\U))$.

The polarizations $S_H$, $S_L$ and $S_V$ of the variations of Hodge structure
$\H$, $\bL$ and $\V$ over $\A_g$ therefore give rise to cohomology classes
$$
\phi_H \in H^2(J(\H)),\quad \phi_L \in H^2(J(\bL)) \text{ and }
\phi_V \in H^2(J(\V)).
$$
Denote their parallel, canonical translation invariant representatives by
$\w_H$, $\w_L$ and $\w_V$.

Recall that $J(\H) = \J_g^c$, the universal jacobian over $\M_g^c$ and that
$\eta_g \in H^{2g}(\J_g^c)$ denotes the Poincar\'e dual of the $0$-section $Z_g$
of $\J_g^c$. A standard and elementary computation shows that if $C$ is a smooth
projective curve of genus $g$, then
$$
\int_{\Jac C} \w_H^g = g!
$$
Combining this with Proposition~\ref{cpact_tori} yields the following result,
which can also be deduced from \cite[Cor.~2.2]{voisin}.

\begin{proposition}
\label{prop:res_phi}
The class $\eta_g$ of the zero section of $\J_g^c$ in $H^{2g}(\J_g^c)$ is
$\phi_H^g/g!$ \qed
\end{proposition}

Denote the restriction of $\J_g$ to $\Mbar_g-\D_0^\sing$ by $\J_g'$. Zucker's
Theorem \cite{zucker} implies that every normal function section $\mu$ of $\J_g$
defined over $\M_{g,n}^c$ extends to a section (also denoted $\mu$) of $\J_g'$
defined over $\Mbar_{g,n}-\D_0^\sing$.

\begin{proposition}
\label{prop:phi_extends}
The class $\phi_H \in H^2(\J_g^c)$ extends naturally to a class $\phihat_H \in
H^2(\J_g')$. It is characterized by the property
$$
[e]^\ast \phihat_H = e^2 \phihat_H\text{ for all } e\in\Z
$$
and has the property that $\widehat{\mu^\ast \phi_H} = \mu^\ast \phihat_H$ for
all normal function sections $\mu$ of $\J_g^c$ defined over $\M_{g,n}^c$.
\end{proposition}

\begin{proof}[Sketch of Proof]
The first rational cohomology of the smooth finite orbi-covering $\Mbar_{g-1,2}$
of $\D_0$ vanishes. The Gysin sequence thus gives an exact sequence
$$
\Q\d_0 \to H^2(\J_g') \to H^2(\J_g^c) \to 0
$$
of rational cohomology. For each integer $e>1$, the endomorphism $[e]$ of
$\J_g'$ induces an action on this sequence. It acts trivially on the left-hand
term. Since $[e]^\ast \phi_H = e^2 \phi_H$ in $H^2(\J_g^c)$, it follows that
$\phi_H$ has a unique lift $\phihat_H$ to $H^2(\J_g')$ with the property that
$[e]^\ast\phihat_H = e^2\phihat_H$. The restriction of $\phihat_H$ to the zero
section $\Mbar_g-\D_0^\sing$ of $\J_g'$ vanishes as $[e]$ preserves the zero
section and acts trivially on $H^2(\Mbar_g-\D_0^\sing)$, and since
$[e]^\ast\phihat_H = e^2\phihat_H$.

To complete the proof, we need several facts about biextension line bundles.
Suppose that $f : Y \to X$ is a morphism of smooth varieties and that $\U$ is a
VHS over $X$ polarized by $S$. Then the constructions of
\cite{hain:biext} imply that
one has a commutative diagram
$$
\xymatrix{
\Bhat_Y \ar[r]^{f_\B}\ar[d] & \Bhat_X\ar[d] \cr
J(f^\ast\U) \ar[r]^{J(f)}\ar[d] & J(\U) \ar[d] \cr
Y \ar[r]^f & X
}
$$
of biextension line bundles, where $f_\B$ is a morphism of metrized line
bundles. This implies that $J(f)^\ast\phi_S = \phi_{f^\ast S}$. The next fact,
which follows from Lear's Theorem, is that if $X'$ and $Y'$ are smooth varieties
in which $X$ and $Y$ are Zariski dense, and where $X'-X$ is a smooth divisor in
$X'$, and if $\d$ is a normal function section of $J(\U) \to X$, then
\begin{equation}
\label{eqn:nat}
\widehat{(f^\ast\d)^\ast\phi_{f^\ast S}} = f^\ast(\widehat{\d^\ast \phi_S})
\in H^2(Y').
\end{equation}

We now apply this with $X=\J_g^c$, $X'=\J_g'$, $Y = \M_{g,n}^c$ and
$Y'=\Mbar_{g,n}-\D_0^\sing$. The variation $\U$ is the standard variation $\H$,
so that $\J(\U) = \J_g^c\times_{\M_g^c}\J_g^c$. Important here is the fact that
$\J_g^c$ is an algebraic variety. The normal function $\d$ will be the diagonal
section of $\J_g^c \times_{\M_g^c}\J_g^c \to \J_g^c$. Finally, $f : Y \to X$
will be a normal function $\mu : \M_{g,n}^c \to \J_g$, which is a morphism as
$\J_g^c \to \M_g^c$ is a family of (semi)-abelian varieties:
$$
\xymatrix{
\J_{g,n}' \ar[d] \ar[r]^(.43){(\D,\mu)} &
\J_g'\times_{\Mbar_g}\J_g' \ar[d]_{\pi_1}
\cr
\Mbar_{g,n}-\D_0^\sing \ar[r]_(.55)\mu \ar@/^1pc/[u]^\mu &
\J_g'\ar@/_1pc/[u]_\D
}
$$
The class $\phi_S \in H^2(\J_g'\times_{\Mbar_g}\J_g')$ is $\pi_2^\ast \phi_H$,
where $\pi_1$ and $\pi_2$ are the two projections $\J_g'\times_{\Mbar_g}\J_g'\to
\J_g'$. It is now a tautology that
$
\widehat{\Delta^\ast\phi_S} = \phihat_H \in H^2(\J_g').
$
This and the naturality statement (\ref{eqn:nat}) now imply that
$\widehat{\mu^\ast\phi_H} = \mu^\ast(\widehat{\D^\ast\phi_S}) =
\mu^\ast\phihat_H$.
\end{proof}

It is important to note that Proposition~\ref{prop:res_phi} does not hold over
$\Mbar_g$. This is because the restriction of $\phihat_H$ to the zero section
vanishes, whereas the conormal bundle of the zero section is the Hodge bundle,
whose top Chern class is non-trivial.

We shall also need the invariant inner product $\Delta$ on $\H\oplus\H$ that
is defined by
$$
S_\Delta\big((u_1,v_1),(u_2,v_2)\big) = S_H(u_1,v_2)-S_H(u_2,v_1)
$$
Even though it preserves the Hodge filtration, it is not a weak polarization as
can be seen by restricting it to the diagonal of $H\oplus H$ (where it is
positive) and anti-diagonal (where it is negative). Denote the associated
cohomology class in $H^2(\J(\H\oplus\H))$ by $\phi_\Delta$. The class $\phi_\D$
extends naturally to a class $\phihat_\D \in H^2(\J_g'\times_{\Mbar_g}\J_g')$.
The proof is similar to that of Proposition~\ref{prop:phi_extends} and is left
to the reader.

\subsection{The classes $\K^\ast\phi_H$, $\nu^\ast\phi_V$, $\nutilde^\ast\phi_L$
and $(\K\times\K)^\ast\phi_\D$}

Pulling back the classes $\phi_H$, $\phi_V$, $\phi_L$ and $\phi_\D$ along the
normal functions $\K$, $\nu$, $\nutilde$ and the normal function section
$\K\times\K$
of
$$
J(\H\oplus\H) = J(\H)\times_{\M_{g,2}}J(\H) \to \M_{g,2}
$$
defined by
$$
\K\times \K : [C;x_1,x_2] \mapsto (\K(x_1),\K(x_2)) \in \Jac(C)\times \Jac(C).
$$
we obtain rational cohomology classes $\K^\ast\phi_H \in H^2(\M_{g,1}^c)$,
$\nu^\ast\phi_V \in H^2(\M_g^c)$, $\nutilde^\ast \phi_L \in H^2(\M_{g,1}^c)$,
and $(\K\times\K)^\ast \phi_\D \in H^2(\M_{g,2}^c)$. Lear's Theorem implies (via
Cor.~\ref{cor:extension}) that they extend naturally to classes
$\widehat{\nu^\ast\phi_V} \in H^2(\Mbar_g)$, $\widehat{\nutilde^\ast \phi_L} \in
H^2(\Mbar_{g,1})$, and $\widehat{(\K\times\K)^\ast \phi_\D} \in
H^2(\Mbar_{g,1})$.

\begin{proposition}
\label{prop:reln}
If $g\ge 2$, then
$(g-1)\widehat{\nutilde^\ast \phi_L} =
\widehat{\nu^\ast\phi_V} + \widehat{\K^\ast \phi_H} \in H^2(\Mbar_{g,1})$.
\end{proposition}

\begin{proof}
This follows immediately from Lemma~\ref{lem:reln} and the fact \cite{pulte}
that $c\circ \nutilde = \K$ and $p \circ \nutilde = \nu$.
\end{proof}

Our next task is to compute each of these classes. First we need to fix notation
for the natural classes in $H^2(\Mbar_{g,n})$.

\section{Divisor Classes} 

Denote the set $\{x_1,x_2,\dots,x_n\}$ of marked points by $I$. Denote the
relative dualizing sheaf of the universal curve $\pi :\cC \to \Mbar_{g,n}$ over
$\Mbar_{g,n}$ by $w$. Its pushforward $\pi_\ast\w$ is locally free of rank $g$.
Recall that $\lambda_1$ denotes the first Chern class of the Hodge bundle
$\pi_\ast w$. The classes $\psi_j,\ x_j\in I$ are defined by
$$
\psi_j := x_j^\ast c_1(w).
$$
When $n=1$, we will denote $\psi_1$ by $\psi$. Note that this definition is
different from the standard definition. With this definition the $\psi$ classes
are natural with respect to the forgetful maps $\Mbar_{g,n+1} \to \Mbar_{g,n}$.

Each component of the boundary divisor of $\M_{g,n}$ in $\Mbar_{g,n}$ has as
its generic point a stable $n$-pointed curve of genus $g$ with exactly one
node. These components are:
\begin{itemize}

\item $\Delta_0$: the generic point is an irreducible, geometrically connected
$n$-pointed curve with one node;

\item $\Delta_0^P$, where $P$ is a subset of $I$ with $|P|\ge 2$: the generic
point is a reducible curve with two geometrically connected components joined at
a single node, one of which has genus 0 --- the points in $P$ lie on the genus 0
component minus its node, the remaining points $P\c := I - P$ lie on the other
(genus $g$) component minus the node;

\item $\Delta_{h}^{P}$, where $0 < h < g$ and $P\subseteq I$ (possibly empty):
the generic point is a reducible curve with exactly one node and two
geometrically connected irreducible components, one of genus $h$, the other of
genus $g-h$; the points in $P$ lie on the genus $h$ component minus the node,
and the rest $P\c$ lie on the other component minus the node. Note that
$\Delta_{g-h}^{P\c} = \Delta_h^P$.

\end{itemize}

Denote the classes of the divisors
$$
\Delta_0,\ \Delta_0^P\ (P \subseteq I,\ |P|\ge 2),\ \Delta_h^P\
(0 < h < g,\ P\subseteq I)
$$
by
$$
\d_0,\ \d_0^P\ (P \subseteq I,\ |P|\ge 2),\ \d_h^P\ (0 < h < g,\
P\subseteq I),
$$
respectively. It is well-known that the classes
$$
\lambda_1,\ \psi_j\ (x_j\in I),\ \d_0,\ \d_0^P\ (P \subseteq I,\ |P|\ge 2),\
\d_h^P\ (0 < h \le g/2,\ P\subseteq I)
$$
comprise a basis of $H^2(\Mbar_{g,n})$.

One also has the Miller-Morita-Mumford classes
$$
\k_j := \pi_\ast c_1(w)^{j+1} \in H^{2j}(\Mbar_g)
$$
which are defined for $j\ge 1$. It follows from Grothendieck-Riemann-Roch (cf.\
\cite{mumford}) that $\kappa_1 = 12 \lambda_1 - \d$, where
$$
\d = \sum_{j=0}^{\lfloor g/2\rfloor} \d_j \in H^2(\Mbar_g).
$$
Define $\k_j, \d \in H^\dot(\Mbar_{g,n})$ to be the pullbacks of $\k_j,\d\in
H^\dot(\Mbar_g)$ under the natural morphisms $\Mbar_{g,n}\to\Mbar_g$. We thus
have the alternate basis
$$
\k_1,\ \psi_j\ (x_j\in I),\ \d_0,\ \d_0^P\ (P \subseteq I,\ |P|\ge 2),\
\d_h^P\ (1\le h \le g/2,\ P\subseteq I)
$$
of $H^2(\Mbar_{g,n})$.

\begin{remark}
All of these divisor classes can be regarded as classes in $H^2(\Mbar_{g,n})$
or in $CH^1(\Mbar_{g,n})$. Note that these two groups are isomorphic, so that
any relation between divisor classes that holds in cohomology also holds in the
Chow ring.
\end{remark}

\section{Formulas for $\nu^\ast \phi_V$ and $\K^\ast\phi_H$}
\label{sec:formulas}

The computation of $F_\dd^\ast \eta_g$ will be reduced to the computation of the
pullbacks of the classes $\phi_H$ and $\phi_\D$ along the normal functions
$\K_j$ and $\delta_{i,j}$. In this section we compute these basic classes. The
formulas reflect the structure of Torelli groups.

The Moriwaki divisor is the class $\widehat{\nu^\ast\phi_V}$:

\begin{theorem}[{\cite[Thm.~1.3]{hain-reed:arakelov}}]
\label{thm:nu}
If $g\ge 2$, then
$$
2\widehat{\nu^\ast\phi_V} =
(8g+4)\lambda_1-g\delta_0-4\sum_{h=1}^{[g/2]}h(g-h)\d_h \in H^2(\Mbar_g,\Z).
$$
\end{theorem}

The case $g\ge 3$ is \cite[Thm.~1.3]{hain-reed:arakelov}. When $g=2$, $\nu=0$
and the result follows from Mumford's computation \cite{mumford} that
$10\lambda_1 = \delta_0 + 2\delta_1$. Suitably interpreted, it holds in genus 1
as $12\lambda_1=\delta_0$ in $H^2(\Mbar_{1,1},\Z)$.

Proposition~\ref{prop:phi_extends} implies that $\widehat{\K^\ast\phi_H}=
\K^\ast \phihat_H$ and $\widehat{(\K\times\K)^\ast \phi_\D} = (\K\times\K)^\ast
\phihat_\D$.

\begin{theorem}
\label{thm:phi}
If $g\ge 2$, then
$$
2\K^\ast \phihat_H =
4g(g-1)\psi - \k_1 - \sum_{h=1}^{g-1} (2h-1)^2\, \d_{g-h}^{\{x\}}
\in H^2(\Mbar_{g,1},\Z).
$$
\end{theorem}

This result holds trivially in genus $1$ as $\K\equiv 0$ and because $\k_1 =
12\lambda_1 - \d_0 = 0$.

\begin{proof}[Sketch of Proof]
Since $\k_1 = 12\lambda_1 - \d$,
$$
\k_1 + \sum_{h=1}^{g-1} (2h-1)^2 \d_{g-h}^{\{x\}}
=
12\lambda_1 - \d_0 + \sum_{h=1}^{g-1} 4h(h-1)\d_{g-h}^{\{x\}}.
$$
So it suffices to prove that
$$
2\K^\ast \phihat_H =
4g(g-1)\psi - 12\lambda_1 + \d_0 - \sum_{h=1}^{g-1} 4h(h-1)\d_{g-h}^{\{x\}}.
$$
This formula, modulo boundary terms, was proved by Morita
\cite[(1.7)]{morita:jacobians}. Another proof is given in
\cite[Thm.~1]{hain-reed:chern}.

The coefficient of $\d_h^{\{x\}}$ is computed using the method of
\cite[\S11]{hain-reed:arakelov}. The Torelli group $T_g$ is replaced by the
Torelli group $T_{g,1}$ associated to a 1-pointed surface. Instead of taking $V
= \Lambda^3 H /(\theta\wedge H)$, we take it to be $H$. The quadratic form $q$
(which is denoted $S_V$ in this paper) is replaced by $c^\ast S_H$, where
$c:\Lambda^3 H \to H$ is the contraction (\ref{eqn:contraction}). We sketch the
monodromy computation using the notation of \cite[\S11]{hain-reed:arakelov}.

The coefficient of $\d_h^{\{x\}}$ is $-\tauhat(\sigma_h)$, where $\sigma_h$ is a
Dehn twist about a separating simple closed curve that divides a pointed, genus
$g$ reference surface into a surface of genus $h$ (that does not contain the
point) and a surface of genus $g-h$, and where $\tauhat$ is a representation of
$T_{g,1}$ into the Heisenberg group associated to $(H,S_H)$.

There is a symplectic basis $a_1,b_1,\dots, a_g,b_g$ of $H_\Z$ such that
$a_1,b_1,\dots,a_h,b_h$ is a basis of $H'$, the first homology of the genus $h$
subsurface. Set $\w' = a_1\wedge b_1 + \dots + a_h\wedge b_h$, the symplectic
form of $H'$. If $u\in H$, then $c(u\wedge \w') = (h-1)u$. Thus
\begin{multline*}
\tauhat_h(\sigma_h)
= \frac{8}{2h-2}\sum_{j=1}^h S_H\big(c(a_j\wedge \w'),c(b_j\wedge\w')\big) \cr
= \frac{8(h-1)^2}{2h-2} \sum_{j=1}^h S_H(a_j,b_j)
= 4h(h-1).
\end{multline*}

It remains to compute the coefficient of $\d_0$. The most direct way to compute
it is by restricting to a curve in the hyperelliptic locus. First note that if
$C$ is a hyperelliptic curve and $x \in C$ is a Weierstrass point, then $\K(C,x)
= 0$. Call such a pair $(C,x)$ a {\em hyperelliptic pointed curve}. Suppose
that  $T$ is a smooth, complete curve and that $f : T \to \Mbar_{g,1}$ is a
morphism where $f(t)$ is the moduli point of an irreducible hyperelliptic
pointed curve for each $t\in T$.\footnote{For example, we can take $T = \P^1$
and $f$ the morphism associated to the family
$$
v^2 = (u-t)u\prod_{j=1}^{2g}(u-a_j),
$$
where $t\in\C$ and the $a_j$ are distinct non-zero complex numbers. A section of
Weierstrass points is given by $x=(0,0)$.} The normal function $f^\ast\K$
vanishes identically on $T$, which implies that $f^\ast\K^\ast\Bhat$ is trivial
as a metrized line bundle over $T-f^{-1}\D_0$. Its extension as a metrized line
bundle to $T$ is therefore trivial. This implies the vanishing of
$$
f^\ast \K^\ast\phihat_H \in H^2(T).
$$
On the other hand, standard techniques can be used to show that
$$
f^\ast(8\lambda_1 + 4g\psi - \d_0) = 0.
$$
The Cornalba-Harris relation \cite[Prop.~4.7]{cornalba-harris} implies that
$$
f^\ast\big((8g+4)\lambda_1 - g\d_0 \big) = 0 \in \Pic T.
$$
It follows that
$$
f^\ast(4g(g-1)\psi - 12\lambda_1 + \d_0)
= (g-1)f^\ast(8\lambda_1 + 4g\psi - \d_0)
-f^\ast\big((8g+4)\lambda_1 - g\d_0 \big) = 0.
$$
These two facts together imply that the coefficient of $\d_0$ in
$\K^\ast\phihat_H$ is $1$.
\end{proof}

Since $\kappa_1 = 12\lambda_1 - \delta$,  Theorems~\ref{thm:nu} and
\ref{thm:phi} and Proposition~\ref{prop:reln} imply the following result when
$g\ge 2$. The case $g=1$ follows from the fact that $\psi = \lambda_1$ and the
well-known relation $\d_0=12\lambda_1$ in $\Pic\Mbar_{1,1}$.

An immediate consequence of Lemma~\ref{lem:reln} and the two previous results is
the following formula for $\widehat{\nutilde^\ast\phi_L}$.

\begin{corollary}
For all $g\ge 1$,
$$
2\widehat{\nutilde^\ast \phi_L} =
8\lambda_1 + 4g\psi - \d_0 - 4\sum_{h=1}^{g-1} h\delta_{g-h}^{\{x\}}
\in H^2(\Mbar_{g,1},\Z).
$$
\end{corollary}

The next result is needed in the solution of Eliashberg's problem.

\begin{theorem}
\label{phi2}
If $g \ge 2$, then in $H^2(\Mbar_{g,2},\Z)$ we have
\begin{multline*}
(\K\times\K)^\ast \phihat_\Delta = (2g-2)(\psi_1 + \psi_2) - \k_1
- (2g-2)^2\d_0^{\{x_1,x_2\}} \cr
- \sum_{h=1}^{g-1} (2h-1)^2\, \d_{g-h}^{\{x_1,x_2\}}
+ (2h-1)(2(g-h)-1)\big(\d_h^{\{x_1\}}+\d_{g-h}^{\{x_2\}}\big)/2.
\end{multline*}
\end{theorem}

Note that in this and subsequent formulas, we will often sum from $h=1$ to
$h=g-1$ and over all subsets $P$ of $I$. Because of this, some terms will appear
twice as $\d_h^P = \d_{g-h}^P$. We do this to emphasize the symmetry of the
formulas and to facilitate later computations.

\begin{proof}
Modulo the coefficients of the $\d_h^{\{x_1\}}$, this formula can be computed
\begin{enumerate}

\item by restricting $(\K\times\K)^\ast\phi_\D$ to any
fiber $C\times C$ and

\item from Theorem~\ref{thm:phi} by restricting to the diagonal $\Mbar_{g,1} \to
\Mbar_{g,2}$, noting that the restriction of $\phihat_\Delta$ to the diagonal
$\J_g$ of $\J_g\times_{\M_{g,2}^c}\J_g$ is $2\phihat_H$.

\end{enumerate}
These computations are straightforward, once one notes that the divisor
$\D_0^{\{x_1,x_2\}}$ is the ``diagonal'' $\M_{g,1}^c \to \M_{g,2}^c$ in
$\M_{g,2}^c$ and that the Chern class of its normal bundle is $\psi$. If we
restrict to a single curve $C$, then in $H^2(C\times C)$
\begin{align*}
(\K\times\K)^\ast \phihat_\Delta 
&= (2g-2)^2 \sum_{j=1}^g 
\big(a_j^{(1)}\wedge b_j^{(2)}- b_j^{(1)}\wedge a_j^{(2)}\big) \cr
&= (2g-2)^2
\big([\text{point}]^{(1)} + [\text{point}]^{(2)} - [\text{diagonal}]\big) \cr
&= (2g-2)(\psi_1+\psi_2) - (2g-2)^2\d_0^{\{x_1,x_2\}}.
\end{align*}
Here $a_1,\dots, b_g$ is a symplectic basis of $H_1(C)$ and, for $x\in H^1(C)$,
$x^{(k)}$ denotes the pullback of $x$ under the $k$th projection $p_k : C^2 \to
C$.

It remains to compute the coefficient of $\d_h^{x_1}$ when $0<h<g$. This we do
using a test curve suggested by Sam Grushevsky. Since $\phihat_\D$ is invariant
when the two factors of $\J_g^c\times_{\M_{g,2}^c}\J_g^c$ are swapped, it
follows that the formula for $(\K\times\K)^\ast\phihat_\D$. is symmetric in
$x_1$ and $x_2$. Since $\d_h^{\{x_1\}} = \d_{g-h}^{\{x_2\}}$, the formula is
also invariant when $h$ is replaced by $g-h$. We therefore conclude that
\begin{multline}
\label{eqn:expression}
(\K\times\K)^\ast \phihat_\Delta = (2g-2)(\psi_1 + \psi_2) - \k_1
- (2g-2)^2\d_0^{\{x_1,x_2\}} \cr
- \sum_{h=1}^{g-1} (2h-1)^2\, \d_{g-h}^{\{x_1,x_2\}}
+ \sum_{h=1}^{g-1} c_h\d_h^{\{x_1\}}
\end{multline}
where $c_h = c_{g-h}$.

Suppose that $0<h<g$. Fix pointed smooth projective curves $(C',P')$ and
$(C'',P'')$ with $g(C')=h$ and $g(C'') = g-h$. Let $C$ be the nodal genus $g$
curve with three components $C'$, $C''$ and $\P^1$, where $C'$ is attached to
$\P^1$ by identifying $P'\in C'$ with $0\in \P^1$ and $P''\in C''$ with $\infty
\in \P^1$.  For $t\in \P^1-\{0,1,\infty\}$ let $C_t$ be the stable $2$-pointed
curve $(C;1,t)$. The closure $T$ of the curve
$$
\P^1 - \{0,1,\infty\} \to \M_{g,2}^c,\quad t \mapsto [C_t]
$$
\begin{figure}[!htb]
\epsfig{file=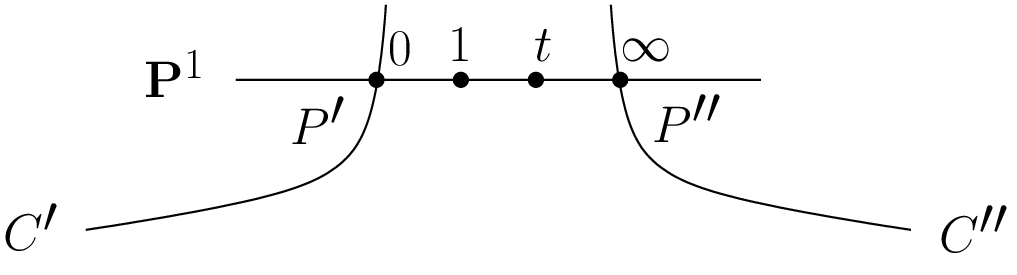, width=3in}
\caption{The $2$-pointed curve $C_t$}
\label{fig:test_curve}
\end{figure}

is a copy of $\Mbar_{0,4} \cong \P^1$ imbedded in $\M_{g,2}^c$.

The restriction of $\K\times\K$ to $T$ takes the constant value
$$
\big((2h-2)P'-K_{C'},(2(g-h)-2)P''-K_{C''}\big)
\in \Jac C'\times \Jac C'' \cong \Jac C,
$$
which implies that $(\K\times\K)^\ast \phi_\D = 0$. The coefficient $c_h$ is
computed by evaluating the right hand side (RHS) of (\ref{eqn:expression}) on
$T$.

The curve $T$ is contained in $\D_h^{\{x_1,x_2\}} \cap \D_{g-h}^{\{x_1,x_2\}}$
and intersects the three boundary divisors $\D_0^{\{x_1,x_2\}}$,
$\D_h^{\{x_1\}}$ and $\D_h^{\{x_2\}}$ transversely in three distinct points.
These are the three boundary points of $\Mbar_{0,4} \cong T$. It does not
intersect any other boundary divisors. Consequently,
$$
\int_T \d_0^{\{x_1,x_2\}} = \int_T \d_h^{\{x_1\}} = \int_T \d_h^{\{x_2\}} = 1.
$$

The projection formula can be used to evaluate the other terms of the RHS of
(\ref{eqn:expression}) on $T$. Let
$$
q: \Mbar_{g,2} \to \Mbar_g, \quad p_j : \Mbar_{g,2} \to \Mbar_{g,1}\quad j =1,2
$$
denote the natural projections, where $p_j([C;x_1,x_2]) = [C,x_j]$. Note that
$q$ and the $p_j$ collapse $T$ to a point. Since $\psi_j = p_j^\ast \psi$ and
$\kappa_1 = q^\ast \kappa_1$, the projection formula implies that
$$
\int_T \kappa_1 = \int_T q^\ast \kappa_1 = \int_{q_\ast T} \kappa_1 = 0
\text{ and }
\int_T \psi_j = \int_T p_j^\ast \psi = \int_{p_j\ast T} \psi = 0.
$$
Since $p_1^\ast \d_h^{\{x\}} = \d_h^{\{x_1,x_2\}} + \d_h^{\{x_1\}}$, the
projection formula implies that
$$
\int_T \d_h^{\{x_1,x_2\}} = - \int_T \d_h^{\{x_1\}} = -1.
$$
Similarly, $\int_T \d_{g-h}^{\{x_1,x_2\}} = -1$.

Evaluating the expression (\ref{eqn:expression}) on $T$ we obtain
$$
0 = 0 + 0 - (2g-2)^2 + (2h-1)^2 + \big(2(g-h)-1\big)^2 + c_h + c_{g-h}.
$$
Since $c_h = c_{g-h}$, this implies that $c_h = (2h-1)\big(2(g-h)-1\big)$.
\end{proof}

\section{Solution of Eliashberg's Problem over $\M_{g,n}^c$ when $g>1$}
\label{sec:cpt_type}

In this section, we solve Eliashberg's problem over $\M_{g,n}^c$ when $g > 2$.
A complete solution in the genus 1 case is given in the following section. The
solution in genus $>1$ is a direct consequence of Proposition~\ref{prop:res_phi}
and Theorem~\ref{pullback_phi} below. Related work on Eliashberg's problem has
been obtained independently by Cavalieri and Marcus \cite{cavalieri-marcus} via
Gromov-Witten theory.

Fix an integral vector $\dd=(d_1,\dots, d_n)$ with $\sum_j d_j=0$. As in the
introduction, we have the section
$$
F_\dd : \M_{g,n}^c \to \J_g
$$
of the universal jacobian defined by
$$
F_\dd : (x_1,\dots,x_d) \mapsto \bigg[\sum_{j=1}^n d_j x_j\bigg] \in \Jac C.
$$
For each subset $P$ of $\{x_1,\dots,x_n\}$, set $d_P = \sum_{j\in P}d_j$. Since
$d_P + d_{P\c} =0$, $d_P^2\d_h^P = d_{P\c}^2\d_{g-h}^{P\c}$.

\begin{theorem}
\label{thm:eliashberg}
If $g\ge 2$, then in $H^{2g}(\M_{g,n}^c)$ we have
$$
F_\dd^\ast \eta_g = \frac{1}{g!}\bigg(\sum_{j=1}^n d_j^2\, \psi_j/2 -
\sum_{P\subseteq I}
\sum_{\{x_j,x_k\}\subseteq P} d_j d_k\, \d_0^P -
\frac{1}{4} \sum_{P\subseteq I}\sum_{h=1}^{g-1} d_P^2\, \d_h^P\bigg)^g.
$$
\end{theorem}

Recall that our definition of the $\psi_j$ differs from the commonly used one.
Here $\psi_j := x_j^\ast c_1(\w)$, where $\w$ is the relative dualizing sheaf of
the universal curve. Note, too, that since $\d_h^P=\d_h^{P^c}$ and since we are
summing from $h=1$ to $h=g-1$ in this and other results in this section, some
boundary divisors occur twice in this expression.

I do not know if this formula holds in the Chow ring. Although this formula
makes sense in $H^{2g}(\Mbar_{g,n}-\D_0^\sing)$, it does not hold there. For
example, when $\dd=0$ the left hand side is the non-trivial class
$(-1)^g\lambda_g$, whereas the right-hand side vanishes. A result of Ekedahl and
van der Geer \cite{ekd-vdg} implies that, in $CH^g(\Mbar_g-\D_0^\sing)$,
$\lambda_g$ is $(-1)^g \zeta(1-2g)$ times a natural class, where $\zeta(s)$
denotes the Riemann zeta function. This suggests that the class
$$
F_\dd^\ast\eta_g -
\frac{1}{g!}(F_\dd^\ast \phihat_H)^g \in CH^g(\Mbar_{g,n}-\D_0^\sing)
$$
should be interesting. In particular, does Proposition~\ref{prop:res_phi} hold
in $CH^g(\J_g^c)$? This makes sense, as $\phi_H = \theta - \lambda_1/2$ in $\Pic
\J_g^c$. (Cf.\ the proof of Theorem~\ref{thm:variant} below.)

\subsection{The approach and reduction}

Denote the pullback of $\J_g$ to $\Mbar_{g,n}$ by $\J_{g,n}$ and its restriction
to $\M_{g,n}^c$ by $\J_{g,n}^c$. We will consider a more general situation.
Namely, we'll assume that $\dd\in\Z^n$ and $m\in\Z$ satisfy
$$
\sum_{j=1}^n d_j = (2g-2)m.
$$
Then one has the section
$$
F_\dd : [C;x_1,\dots,x_n] \mapsto \sum_{j=1}^n d_j x_j - m K_C \in \Jac C
$$
of $\J_{g,n}^c$ over $\M_{g,n}^c$. Our goal is to compute $F_\dd^\ast \phi_H$.

Set
$$
\J_{g,n}^m :=
\underbrace{\J_{g,n}\times_{\Mbar_{g,n}} \times \J_{g,n}\times_{\Mbar_{g,n}}
\cdots
\times_{\Mbar_{g,n}} \J_{g,n}}_m.
$$
Denote its restriction to $\M_{g,n}^c$ by $\J_{g,n}^{c,m}$. Let
$$
\K_n : \Mbar_{g,n} - \D_0^\sing \to \J_{g,n}^n
$$
be the $n$th power of $\K$ --- that is, the section of $\J_{g,n}^{c,n} \to
\M^c_{g,n}$ defined by
$$
\K_n : (C;x_1,\dots,x_n) \mapsto \big(\K(x_1),\K(x_2),\dots,\K(x_n)\big)
\in \big(\Jac C\big)^n.
$$
Note that $\K_2=\K\times \K$.

\begin{proposition}
Define $\dd : \J_{g,n}^n \to \J_{g,n}^n$ by
$$
\dd : (u_1,\dots,u_n) \mapsto (d_1 u_1,\dots,d_n u_n)
$$
and $\trace_n : \J_{g,n}^n \to \J_{g,n}$ by
$$
\trace_n : (u_1,\dots,u_n) \mapsto u_1 + \dots + u_n.
$$
Then the mapping
$$
\xymatrix{
\Mbar_{g,n} -\D_0^\sing \ar[r]^(0.6){\K_n} & \J_{g,n}^n \ar[r]^{\dd} &
\J_{g,n}^n \ar[r]^{\trace_n} & \J_{g,n}
}
$$
equals $(2g-2) F_\dd$. \qed
\end{proposition}

This formula allows the reduction of the computation of $F_\dd^\ast \phi_H$ to
more basic computations. Denote the pullback of $\phi_H$ under the $j$th
projection
$$
p_j : \J_{g,n}^{c,n} \to \J_{g,n}^c
$$
by $\phi_{j,j}$. For $j\neq k$, denote the pullback of $\phi_\Delta$ under
the $(j,k)$th projection
$$
p_{j,k} : \J_{g,n}^{c,n} \to
\J_{g,n}^{c,2} \quad (u_1,\dots,u_n) \mapsto (u_j,u_k)
$$
by $\phi_{j,k}$.

\begin{lemma}
\label{sum}
With notation as above,
$$
\dd^\ast \phi_{j,k} = d_j d_k\, \phi_{j,k} \text{ and }
\trace_n^\ast \phi_H = \sum_{j \le k} \phi_{j,k}.
$$
\end{lemma}

\begin{proof}
Since all of the classes $\phi_{j,k}$ are represented by parallel, translation
invariant forms, to prove the result, it suffices to prove the result in the
cohomology of the jacobian $J := \Jac C$ of a single smooth projective curve
$C$.

Set $J = \Jac C$. Note that the ring homomorphism
$$
H^\dot(J) \to H^\dot(J)^{\otimes n}
$$
induced by the addition map $J^n \to J$ is, in degree 1, given by
$$
x \mapsto \sum_{\substack{a+b=n-1\cr a,b\ge 0}}
1^{\otimes a} \otimes x \otimes 1^{\otimes b}.
$$
The formula for $\trace_n^\ast$ follows using the fact that this is a
ring homomorphism.

The formula for $\dd^\ast$ follows as the map $[e] : J \to J$ is multiplication
by $e$ on $H^1(J)$, and therefore multiplication by $e^k$ on $H^k(J)$. 
\end{proof}

Recall that $\phi_H \in H^2(\J_g^c)$ and $\phi_\D \in
H^2(\J_g^c\times_{\M_g^c}\J_g)$ extend naturally to classes
$$
\phihat_H \in H^2(\J_g')\text{ and }
\phihat_\D \in H^2(\J_g'\times_{\Mbar_g}\J_g'),
$$
where $\J_g'$ denotes the universal jacobian over $\Mbar_g-\D_0^\sing$. Using
the scaling by the action of $(e_1,\dots,e_n) \in \Z^n$ on the cohomology of the
$n$th power of $\J_g' \to \Mbar_g-\D_0^\sing$ one can show that Lemma~\ref{sum}
holds when $\phi_H$ and $\phi_\D$ are replaced by the extended classes
$\phihat_H$ and $\phihat_\D$. We therefore have:

\begin{corollary}
\label{cor:pullback}
If $g\ge 2$, then
$$
(2g-2)^2\,F_\dd^\ast \phihat_H =
\sum_{j=1}^n d_j^2\,\pi_j^\ast \K^\ast \phihat_H
+
\sum_{1\le j<k \le n} d_j d_k\, \pi_{j,k}^\ast(\K\times\K)^\ast \phihat_\Delta
\in H^2(\M_{g,n}^c),
$$
where $\pi_j : \Mbar_{g,n} \to \Mbar_{g,1}$ and $ \pi_{j,k} : \Mbar_{g,n} \to
\Mbar_{g,2}$ denote the natural projections.
\end{corollary}

The following result is the main computation of this section.

\begin{theorem}
\label{pullback_phi}
If $g \ge 2$ and $\dd\in \Z^n$ and $m\in\Z$ satisfy $\sum d_j = (2g-2)m$, then
in $H^2(\Mbar_{g,n})$ we have
\begin{multline*}
F_\dd^\ast\phihat_H = -{m^2}\kappa_1/2 +
\sum_{j=1}^n (d_jm+d_j^2/2)\,\psi_j - \sum_{P\subseteq I}
\sum_{\{x_j,x_k\}\subseteq P} d_j d_k\, \d_0^P \cr
-\frac{1}{4} \sum_{P\subseteq I}\sum_{h=1}^{g-1}
\big(d_P - (2h-1)m\big)^2\, \d_h^P.
\end{multline*}
\end{theorem}

Recall that $\d_h^P=\d_{g-h}^{P^c}$. Note that, in this expression, the
coefficients of $\d_h^P$ and $\d_{g-h}^{P^c}$ are equal as $d_P + d_{P^c} =
(2g-2)m$. Note, too, that one recovers Theorem~\ref{thm:phi} when $n=m=1$ and
$d_1=2g-2$.

Theorem~\ref{thm:eliashberg} follows directly from
Proposition~\ref{prop:res_phi} and the case $m=0$. The corresponding genus 1
statement is proved in the following section. The proof below fails in genus 1
as we cannot divide by $2g-2$.

\begin{proof}
This follows from Theorems~\ref{thm:phi} and \ref{phi2} and
Corollary~\ref{cor:pullback}. First observe that the coefficient of $\k_1$ in
$(2g-2)^2F_\dd^\ast\phihat_H$ is
$$
-\sum_j  d_j^2/2 - \sum_{j<k} d_j d_k = -(d_1 + \cdots + d_n)^2/2
= -(2g-2)^2m^2/2.
$$
Next, the coefficient of $\psi_j$ in $(2g-2)^2\,F_\dd^\ast \phihat_H$ is
\begin{align*}
(2g-2)\sum_{k\neq j} d_jd_k + \frac{1}{2}\, 4g(g-1)d_j^2 
&= (2g-2)\big[(d_1 + \cdots + d_n - d_j)d_j + g\,d_j^2\big] \cr
&= (2g-2)^2 d_j m + (2g-2)(g-1)d_j^2 \cr
&= (2g-2)^2(d_j m +d_j^2/2).
\end{align*}

Since
$$
\pi_{j,k}^\ast \d_0^{\{x_1,x_2\}} = \sum_{\{x_j,x_k\}\subseteq P} \d_0^P,
\quad j \neq k
$$
the coefficient of $\d_0^P$ in the expression for $(2g-2)^2 F_\dd^\ast
\phihat_H$ is
$$
-(2g-2)^2 \sum_{\{x_j,x_k\}\subseteq P} d_jd_k\, \d_0^P.
$$

When $1\le h \le g-1$ and $j\neq k$
$$
\pi_j^\ast \d_h^{\{x\}} = \sum_{x_j \in P} \d_h^P,\quad
\pi_{j,k}^\ast \d_h^{\{x_1,x_2\}} = \sum_{\{x_j,x_k\} \subseteq P} \d_h^P,
\text{ and }
\pi_{j,k}^\ast \d_h^{x_1} = \sum_{\substack{x_j \in P\cr x_k \in P^c}} \d_h^P.
$$
Then, computing formally, we see that the coefficient of $\d_h^P$ in
$(2g-2)^2F_\dd^\ast\phihat_H$ is the coefficient of $\d_h^P$ in 
\begin{multline*}
-\big(2(g-h)-1\big)^2\sum_j d_j^2\, \pi_j^\ast\, \d_h^{\{x\}}/2
-\big(2(g-h)-1\big)^2 \sum_{j<k} d_jd_k\,\pi_{j,k}^\ast\, \d_h^{\{x_1,x_2\}}
\cr
+ (2h-1)\big(2(g-h)-1\big)\sum_{j < k}
d_j d_k \pi_{j,k}^\ast \big(\d_h^{x_1}+ \d_{g-h}^{x_2}\big)/2,
\end{multline*}
which is
\begin{align*}
&-\frac{\big(2(g-h)-1\big)^2}{2}
\bigg(\sum_{x_j \in P} d_j^2 +
\sum_{\substack{x_j,x_k\in P\cr x_j \neq x_k}} d_j d_k\bigg)
+ \frac{(2h-1)\big(2(g-h)-1\big)}{2}
\sum_{\substack{x_j\in P\cr x_k \in P^c}} d_j d_k \cr
&= -\frac{1}{2}\Big(\big(2(g-h)-1\big)^2 d_P^2
-\big(2(g-h)-1\big)(2h-1)d_Pd_{P^c}\Big).
\end{align*}
Since $\d_h^P = \d_{g-h}^{P^c}$, the coefficient of $\d_h^P$ can be chosen (and
will be chosen) to be the average of the formally computed coefficients of
$\d_h^P$ and $\d_{g-h}^{P^c}$, which is
\begin{align*}
&-\frac{1}{4}\Big(\big(2(g-h)-1\big)^2 d_P^2
-\big(2(g-h)-1\big)(2h-1)d_Pd_{P^c}+(2h-1)^2 d_{P^c}^2\Big)
\cr
&= -\frac{1}{4}\Big(\big(2(g-h)-1\big)d_P
-(2h-1)d_{P^c}\Big)^2
\cr
&= -\frac{1}{4}\Big(\big(2(g-h)-1\big)d_P
+(2h-1)\big(d_P -(2g-2)m\big)\Big)^2
\cr
&=  -\frac{(2g-2)^2}{4}\big(d_P -(2h-1)m\big)^2
\end{align*}
as $d_P + d_{P^c}=(2g-2)m$.
\end{proof}

\subsection{Variants}

Theorem~\ref{pullback_phi} can be adapted to establish more general results.
Suppose that $r$ divides $2g-2$ and suppose that $\dd = (d_1,\dots,d_n)\in \Z^n$ and $m\in
\Z$ satisfy
$$
\sum_{j=1}^n d_j = m(2g-2)/r.
$$
Let $f : X \to \Mbar_{g,n}-\D_0^\sing$ be a morphism over which there is a
globally defined $r$th root $\alpha$ of the relative dualizing sheaf. Then one
has the section $E_\dd$ of $\J_{g,n}^c$ over $X$ defined by
$$
E_\dd : x \mapsto - m\alpha + \sum_{j=1}^n d_j x_j \in \Jac C.
$$
where $f(x) = [C;x_1,\dots,x_n]$. As above, $d_P = \sum_{x_j \in P} d_j$.

\begin{theorem}
\label{thm:variant}
The class $E_\dd^\ast \phihat_H \in H^2(X)$ is the pullback along $f$ of the
class
\begin{multline*}
-(m/r)^2\kappa_1/2 +
\sum_{j=1}^n (d_j(m/r)+d_j^2/2)\,\psi_j - \sum_{P\subseteq I}
\sum_{\{x_j,x_k\}\subseteq P} d_j d_k\, \d_0^P \cr
-\frac{1}{4} \sum_{P\subseteq I}\sum_{h=1}^{g-1}
\big(d_P - (2h-1)(m/r)\big)^2\, \d_h^P \in H^2(\Mbar_{g,n}).
\end{multline*}
\end{theorem}

\begin{proof}
Since the diagram
$$
\xymatrix{
& \J_{g,n}\ar[d] \cr
X \ar[r]_(.3)f\ar[ur]^{rE_\dd} & \Mbar_{g,n}-\D_0^\sing \ar@/_1pc/[u]_{F_{r\dd}}
}
$$
commutes (i.e., $F_{r\dd}\circ f = r E_\dd$), and since the extended class
$\phihat_H$ satisfies $[e]^\ast\phihat_H = e^2\phihat_H$, we have
$$
E_\dd^\ast \phihat_H = r^{-2} E_{r\dd}\phihat_H
= r^{-2} f^\ast F_{r\dd}\phihat_H.
$$
The result now follows from Theorem~\ref{pullback_phi}.
\end{proof}

This result can be used to give a partial solution to a problem posed to me by
Joe Harris. Suppose that $\dd = (d_1,\dots,d_n) \in \Z^n$ satisfies
$\sum_{j=1}^n d_j = g-1$. Then one has the section
$$
G_\dd : [C;x_1,\dots,x_n] \mapsto \sum_j d_j x_j \in \Pic^{g-1} C
$$
of the relative Picard bundle
$$
\cP_{g,n} := \Pic^{g-1}_{\cC_{g,n}/\M_{g,n}}
$$
over $\M_{g,n}$. This contains the divisor $W$ of effective divisor classes of
degree $g-1$. The pullback $G_\dd^\ast W$ is the divisor in $\M_{g,n}$
consisting of those $[C;x_1,\dots,x_n]$ where $h^0(C,\sum d_j x_j)>0$. Denote
its closure in $\Mbar_{g,n}$ by $W_\dd$. Harris' problem is to compute the class
of $W_\dd$ in terms of standard classes. This is a subtle problem as
$\cP_{g,n}$ is not separated over $\M_{g,n}^c$.

Although we cannot solve this problem, we can solve the following closely
related problem. Let $X \to \Mbar_{g,n}-\D_0^\sing$ be any dominant morphism on
which there is a globally defined theta characteristic $\alpha$. Denote the
inverse image of $\M_{g,n}^c$ in $X$ by $X^c$ and the inverse image in $X$ of
$\M_{g,n}$ by $X^o$. Denote the universal jacobian over $X$ by $\J_X$ and its
restrictions to $X^c$ and $X^o$ by $\J_X^c$ and $\J_X^o$, respectively. Denote
the pullback of $\cP_{g,n}$ to $X^o$ by $\cP_X$. Then $\alpha$ defines an
isomorphism of $\cP_X$ with $\J_X^o$. Under this isomorphism $G_\dd$ corresponds
to the section
$$
F_\dd : [C;x_1,\dots,x_n] \mapsto -\alpha + \sum_{j=1}^n d_j x_j \in \Jac C
$$
and the pullback of the divisor $W$ to $\cP_X$ corresponds to the divisor
$\Theta_\alpha$ which is defined locally by the theta function
$\vartheta_\alpha(z,\Omega)$ that corresponds to $\alpha$.

The section $F_\dd$ extends to a section of $\J_X$ over $X$ and the divisor
$\Theta_\alpha$ extends to $\J_X$. One therefore has the class
$F_\dd^\ast\Theta_\alpha$ in $H^2(X)$. Its restriction to $H^2(X^o)$ is the
pullback of the class of $W_\dd$.

\begin{theorem}
\label{thm:harris}
The class of $F_\dd^\ast \Theta_\alpha$ in $H^2(X)$ is
\begin{multline*}
\d_0/8 -\lambda_1 +
\sum_{j=1}^n (d_j+d_j^2)\,\psi_j/2 - \sum_{P\subseteq I}
\sum_{\{x_j,x_k\}\subseteq P} d_j d_k\, \d_0^P \cr
-\frac{1}{4} \sum_{P\subseteq I}\sum_{h=1}^{g-1}
\big(d_P^2 - (2h-1)d_P + h^2 - h\big)\, \d_h^P.
\end{multline*}
\end{theorem}

The class $-\d_0/8+F_\dd^\ast\Theta_\alpha$ is the pullback of an integral class
from $\Mbar_{g,n}$ as $\d_h^P = \d_{g-h}^{P^c}$ and because each of $d_j +
d_j^2$, $d_P^2 - (2h-1)d_P$ and $h^2 - h$ is even.

\begin{proof}
The first step is to show that the class $\theta_\alpha$ of $\Theta_\alpha$
satisfies
\begin{equation}
\label{eqn:theta}
\theta_\alpha = \phihat_H + \lambda_1/2 \in H^2(\J_g'),
\end{equation}
where $\J_g'$ denotes the universal jacobian over $\Mbar_g-\D_0^\sing$. Granted
this, the result follows from Theorem~\ref{thm:variant} as
$$
F_\dd^\ast(\theta_\alpha) = F_\dd^\ast(\phihat_H + \lambda_1/2)
= \lambda_1/2 + F_\dd^\ast(\phihat_H) 
$$
and
$$
\k_1 = 12\lambda_1 - \d =
12\lambda_1 - \d_o -\frac{1}{2} \sum_{j=1}^h \sum_{P\subseteq I} \d_h^P.
$$

To prove (\ref{eqn:theta}), first note that the relation
$\theta_\alpha = \phi_H + \lambda_1/2$
holds in $H^2(\J_g^c)$. This is because the restrictions of $\theta_\alpha$ and
$\phi_H$ to each fiber of $\J_g^c \to \M_g^c$ are equal and because the
restriction of $\phi_H$ to the zero section vanishes, while the restriction of
$\theta_\alpha$ to the zero section has class $\lambda_1/2$ as the corresponding
theta null $\vartheta_\alpha(0,\Omega)$ is a modular form of weight $1/2$ for
some finite index subgroup of $\Sp_g(\Z)$. Since $\phihat_H$, $\theta_\alpha$
and $\lambda_1$ are all classes of line bundles, and since $\J_g - \J_g^c$ is
the restriction of $\J_g'$ to $\Delta_0$, it follows that
$$
\theta_\alpha = \phihat_H + \lambda_1/2 + c\d_0 \in H^2(\J_g')
$$
Restricting both sides to the zero section implies that
$$
\lambda_1/2 = 0 + \lambda_1/2 + c\d_0
\in H^2(\Mbar_g-\D_0^\sing) \cong H^2(\Mbar_g),
$$
which implies that $c=0$, as required.
\end{proof}

\section{Eliashberg's Problem in Genus $1$}

The solution of Eliashberg's problem over $\M_{g,n}^c$ in genus $\ge 2$ given in
the previous section fails when $g=1$ as we cannot divide by $2g-2$. However, a
variant of the methods of the previous section gives a complete solution in
genus $1$.

When $g=1$, the class of $F_\dd^\ast\eta_1$ naturally lives in
$H^2(\Mbar_{1,n+1})$ as the locus of indeterminacy $\D_0^\sing$ of $F_\dd$ has
codimension $\ge 2$, whereas $F_\dd \eta_1$ is the class of a divisor, and thus
extends uniquely from $\Mbar_{1,n+1}-\D_0^\sing$ to $\Mbar_{1,n+1}$.

\begin{theorem}
\label{thm:elliptic}
If $\dd=(d_0,\dots,d_n)\in \Z^{n+1}$ satisfies $\sum_j d_j = 0$, then
$$
F_\dd^\ast \eta_1 =
\big(-1+(d_0^2 + \dots + d_n^2)/2\big)\lambda_1
 - \sum_{P\subseteq I}
\sum_{\{x_j,x_k\}\subseteq P} d_j d_k\, \d_0^P
\in H^2(\Mbar_{1,n+1}).
$$
\end{theorem}

The restriction of this class to $\M_{1,n+1}$ has been computed independently by
Cavalieri and Marcus \cite{cavalieri-marcus} using different methods.

Denote the universal elliptic curve $\J_1 \to \Mbar_{1,1}$ by $\cE$. Note that
$\cE = \Mbar_{1,2}-\D_0^\sing$. Its restriction to $\M_{1,1}$ is the universal
elliptic curve $\cE^c$ of ``compact type''. Since $\cE$ and $\Mbar_{1,2}$ differ
in codimension $2$, their second cohomology and Picard groups are isomorphic:
$$
H^2(\Mbar_{1,2}) \cong (\Pic \Mbar_{1,2})\otimes\Q \cong (\Pic \cE)\otimes\Q
\cong H^2(\cE).
$$
These groups are 2-dimensional with basis $\d:=\d_0^{\{x_1,x_2\}}$, the class of
the zero section $D$ of $\cE \to \Mbar_{1,1}$, and $\d_0$, the class of the
fiber over the cusp of $\Mbar_{1,1}$. The class $\lambda_1$ of the Hodge bundle
is $\d_0/12$.

We will deduce Theorem~\ref{thm:elliptic} from the corresponding result for
the $n$th power
$$
\cE^n = \underbrace{\cE\times_{\Mbar_{1,1}} \cdots \times_{\Mbar_{1,1}} \cE}_n.
$$
of the universal elliptic curve over $\Mbar_{1,1}$. A point in $\cE^n$
corresponds to the isomorphism class of an $n$-pointed elliptic curve
$(E;x_0,x_1,\dots,x_n)$ where $x_0 = 0$, the identity.\footnote{Note that an
$n$-pointed elliptic curves is an $(n+1)$-pointed genus $1$ curve.} Let $p :
\cE^n \to \Mbar_{1,1}$ be the canonical projection. For each $\dd =
(d_0,\dots,d_n)$ as above, there is a section of the pullback $p^\ast \cE \to
\cE^n$ of the universal elliptic curve defined by
$$
F_\dd(E;x_0,\dots,x_n) = \sum_{j=0}^n d_j x_j = \sum_{j=1}^n d_j x_j \in E.
$$

For $0\le j < k \le n$ let $D_{j,k}$ be the divisor in $\cE^n$ where $x_j =
x_k$. Denote its class in $H^2(\cE^n)$ by $\d_{j,k}$. Let $\D_0$ be the inverse
image of the moduli point of the nodal cubic under the projection $\cE^n \to
\Mbar_{1,1}$. Then
$$
H^2(\cE^n) \cong (\Pic \cE^n) \otimes \Q =
\Q\d_0 \oplus \bigoplus_{0\le j < k \le n} \Q \d_{j,k}.
$$
The pullback of $\d_{j,k}$ under the morphism $\pi : \Mbar_{1,n+1}-\D_0^\sing
\to \cE^n$ is
$$
\pi^\ast \d_{jk} = \sum_{\{x_j,x_k\}\subseteq P} \d_0^P \in \Pic \Mbar_{1,n+1}.
$$
Since $d_0 + \dots + d_n = 0$, $d_0^2 + \dots + d_n^2 = -2\sum_{0\le j < k \le n
} d_jd_k$. Since $F_\dd^\ast \lambda_1 = \lambda_1$, to prove
Theorem~\ref{thm:elliptic} it suffices to prove that
\begin{equation}
\label{eqn:reduction}
F_\dd^\ast(\d + \lambda_1) =
- \sum_{0\le j<k\le n} d_jd_k (\d_{j,k}+\lambda_1)
\in H^2(\cE^n).
\end{equation}

The key step in the proof of this statement is to show that $\delta + \lambda_1$
is a parallel, translation invariant class. This statement is made precise in
the following lemma. For a positive integer $e$, let $[e] : \cE \to \cE$ be
multiplication by $e$.

\begin{lemma}
\label{lem:phihat}
The class $\phi_H \in H^2(\cE^c)$ extends uniquely to a class $\phihat_H$ in
$H^2(\cE)$ that vanishes on the zero-section $D$. It is given by
$$
\phihat_H = \d + \lambda_1 \in H^2(\cE). 
$$
and is characterized by the two properties
$$
\int_{\D_0} \phihat_H = 1 \text{ and } [e]^\ast \phihat_H = e^2 \phihat_H
\text{ for all integers $e>1$.}
$$
\end{lemma}

\begin{proof}
The exact sequence
$$
0 \to \Q\d_0 \to H^2(\cE) \to H^2(\cE^c) \to 0.
$$
is invariant under $[e]^\ast$, which acts trivially on the kernel and by
multiplication by $e^2$ on the quotient. Since $\phi_H$ spans the right-hand
group, it follows that it has a unique lift $\phihat_H$ to $H^2(\cE)$ with the
property that $[e]^\ast\phihat_H = e^2\phihat_H$.

Since $[e]_\ast [D] = [D]$, we have
$$
(e^2 - 1)\int_D\phihat_H = \int_D [e]^\ast\phihat_H - \int_D\phihat_H
= \int_{([e]_\ast - 1)D} \phihat_H = 0.
$$
Since $e > 1$, this implies the vanishing $\int_D \phihat_H = 0$ of $\phihat_H$
on $D$. Since $\D_0$ is the class of the fiber of $\cE \to \Mbar_{1,1}$,
$\int_{\D_0}\phihat_H = 1$. Since $\d$ and $\d_0$ span $H^2(\cE)$ and since its
intersection pairing is non-singular, these two properties characterize
$\phihat_H$.

To prove that $\phihat_H = \delta + \lambda_1$, it suffices to show that
$(\d+\lambda_1)\cdot \d = 0$ and that $(\d+\lambda_1)\cdot \d_0 = 1$. Since the
Chen class of the normal bundle of $D$ in $\cE$ is $-\lambda_1$, we have
$$
\int_D\d = D^2 =  -\int_D \lambda_1,
$$
which implies that $(\d+\lambda_1)\cdot\d = 0$. Since $\d_0$ is the class of a
fiber of $\cE \to \Mbar_{1,1}$, $\d_0 \cdot \lambda_1 = 0$, so that
$\d_0\cdot(\d+\lambda_1) = \d_0\cdot \d = 1$.
\end{proof}

Each unordered pair $\{j,k\}$ of integers in $[0,n]$ determines a parallel class
$\phihat_{j,k} \in H^2(\cE^n)$ as follows. Let $p_{j,k} : \cE^n \to \cE$ be the
projection that takes $[E:x_0,\dots,x_n]$ to $[E;x_j,x_k]=[E;0,x_k-x_j]$. 
Observe that $D_{j,k} = p_{j,k}^\ast D$. Set $\phihat_{j,k} = p_{j,k}^\ast
\phihat_H$. Lemma~\ref{lem:phihat} implies that $\phihat_{j,k} = \d_{j,k} +
\lambda_1$. Combining this with (\ref{eqn:reduction}), we deduce that, to prove
Theorem~\ref{thm:elliptic}, it suffices to show that
\begin{equation}
\label{eqn:reduction2}
F_\dd^\ast \phihat_H = - \sum_{0\le j<k\le n} d_jd_k \phihat_{j,k}.
\end{equation}

\begin{lemma}
The class $\phi_\D \in H^2(\cE^c\times_{\M_{1,1}}\cE^c)$ extends to a class
$\phihat_\D$ in $H^2(\cE^2)$ with the property that $[e]^\ast \phihat_\D =
e^2\phihat_\D$. Specifically,
\begin{equation}
\label{eqn:phiD}
\phihat_\D =
p_{0,1}^\ast\phihat_H + p_{0,2}^\ast\phihat_H - p_{1,2}^\ast\phihat_H
\in H^2(\cE^2).
\end{equation}
\end{lemma}

\begin{proof}
The class given by (\ref{eqn:phiD}) is clearly an eigenvector of $[e]$ with
eigenvalue $e^2$. To prove the result, it suffices to prove that its restriction
to $(\cE^c)^2 := \cE^c\times_{\M_{1,1}}\cE^c$ is $\phi_\D$. Since the
restriction to $(\cE^c)^2$ of all classes in the formula are represented by
parallel, translation invariant forms, it suffices to check the formula for a
single smooth elliptic curve $E$. Let $a,b$ be a symplectic basis of $H^1(E)$.
Identify $H^\dot(E^2)$ with $H^\dot(E)\otimes H^\dot(E)$. Then a routine
computation shows that the restriction of $p_{1,1}^\ast\phi_H$ to $E^2$ is
$$
(a\wedge b)\otimes 1 + 1\otimes(a\wedge b) - (a\otimes b - b\otimes a)
= \big(p_{0,1}^\ast \phi_H + p_{0,2}^\ast \phi_H - \phi_\D\big)|_{E^2},
$$
as required.
\end{proof}

For each $0\le j < k \le n$ define $\phihat^{j,k}_\D = p_{j,k}^\ast \phihat_\D$.

\begin{corollary}
If $0 < j < k\le n$, then
$\phihat^{j,k}_\D = \phihat_{0,j} + \phihat_{0,k} - \phihat_{j,k}$.
\end{corollary}

To prove (\ref{eqn:reduction2}), factor $F_\dd : \cE^n \to \cE$ as follows:
$$
\xymatrix{
\cE^n \ar[r]_(0.45){(d_1,\dots,d_n)} \ar@/^1pc/[rr]^{F_\dd} &
\cE^n \ar[r]_{\trace_n} & \cE 
}
$$
where $(d_1,\dots,d_n) : [E;0,x_1,\dots,x_n] \mapsto [E:0,d_1x_1,\dots,d_nx_n]$.
Note that the formula
$$
\trace_n^\ast \phihat_H
= \sum_{j=1}^n \phihat_{0,j} + \sum_{1\le j < k\le n} \phihat_\D^{j,k}
$$
holds in $H^2(\cE^n)$ as it holds when restricted to $(\cE^c)^n$, because both
sides are eigenvectors of $[e]$ with eigenvalue $e^2$, and because both sides
vanish on the divisor $x_1 + \dots + x_n = 0$ by an argument similar to the one
used in the proof of Lemma~\ref{lem:phihat}.

Adapting the arguments of the previous section to this case, we see that
\begin{align*}
F_\dd^\ast \phihat_H &=
\sum_{j=1}^n d_j^2 \phihat_{0,j}
+ \sum_{1\le j < k\le n} d_jd_k \phihat_\D^{j,k} \cr
& = \sum_{j=1}^n d_j^2 \phihat_{0,j} +
\sum_{1\le j < k\le n} d_jd_k (\phihat_{0,j} + \phihat_{0,k} - \phihat_{j,k})\cr
& = \sum_{j=1}^n \big(d_j^2 + (d_1 + \dots + d_n - d_j)d_j\big)\d_{0,j}
- \sum_{1\le j < k\le n} d_jd_k \phihat_{j,k} \cr
& = \sum_{j=1}^n -d_0d_j \d_{0,j} -
 \sum_{1\le j < k\le n} d_jd_k \phihat_{j,k} \cr
 & = -\sum_{0\le j < k\le n} d_jd_k \phihat_{j,k}.
\end{align*}

\section{Normal Functions and Positivity}

Suppose that $\V$ is a polarized variation of Hodge structure over
$X$ of weight $-1$ endowed with a weak polarization $S$.

\begin{theorem}
\label{positivity}
If $\nu : X \to J(\V)$ is a normal function, then $\nu^\ast\w_S$ is a
non-negative $(1,1)$-form on $X$.
\end{theorem}

\begin{proof}
As previously remarked, $J(\V)$ is isomorphic, as a bundle of 
tori, to $\V_\R/\V_\Z$. The normal function $\nu$ thus corresponds
to a section $s : X \to \V_\R/\V_\Z$. Locally
this lifts to a section (also denoted by s) of $\V_\R$. We can view this
as as section of
$$
\cV := \V\otimes \O_X
$$
This is a flat holomorphic vector bundle. Denote its Hodge filtration by
$$
\cV \supseteq \cdots \supseteq \F^p \supseteq \F^{p+1} \supseteq \cdots
$$
Each $\F^p$ is a holomorphic sub-bundle of $\cV$. Since $\V$ has weight $-1$, it
splits as the sum
$$
\cV = \F^0 \oplus \Fbar^0,
$$
where $\Fbar^p$ denotes the complex conjugate of $\F^p$ in $\cV$. Note that
$\F^0$ is a holomorphic sub-bundle, while $\Fbar^0$ is not, in general,
holomorphic. Decompose $s$ as
$$
s = p + n
$$
where $p$ is a smooth section of $\F^0$ and $n$ is a smooth section of
$\Fbar^0$. Since $s$ is real, $p$ and $n$ are complex conjugates of each
other.

We can compute the differentials of $s$, $n$ and $p$ with respect to the
flat structure on $\V_\R$. Since $\w_S$ is parallel, we have
$$
\nu^\ast \w_S = s^\ast\w_S = S(ds,ds) =
S(\del s + \delbar s,\del s + \delbar s) = 2 S(\del s,\delbar s).
$$
This is clearly of type $(1,1)$.

Next we prove that $\nu^\ast \w_S$ is non-negative.
Since a 2-form is positive if and only if it is positive on every
holomorphic arc in $X$, we may assume that $X$ is the unit disk. Let
$t$ be a holomorphic coordinate in $X$.

The Griffiths infinitesimal period relation for normal functions implies that
$\partial f/\partial t \in \F^{-1}$ for any smooth local lift $f : X \to \cV$ of
the normal function. Here, and in what follows, the partial derivatives are
taken with respect to the natural flat connection on $\cV$. Since $n$ and $p+n$
are both smooth local lifts of $\nu$,
\begin{equation}
\label{griff_inf}
\frac{\partial p}{\partial t}\in \F^{-1} \text{ and }
\frac{\partial n}{\partial t}
\in \F^{-1}.
\end{equation}

Since $\F^0$ is a holomorphic sub-bundle of $\cV$, $\partial p/\partial \tbar
\in \F^0$. Since $n$ is the conjugate of $p$, we have $\partial n/{\partial t}
\in \Fbar^0$. Combining this with Griffiths infinitesimal period relation
(\ref{griff_inf}), we conclude that ${\partial n}/{\partial t} \in \F^{-1}\cap
\Fbar^0$ and, by taking complex conjugates, that ${\partial p}/{\partial \tbar}
\in \F^0\cap \Fbar^{-1}$.

Next compute the pullback of $\w_S$: as above,
\begin{align*}
\nu^\ast \w_S
&= 2 S(\del s,\delbar s) \cr
&= 2 S(\del n + \del p, \delbar n + \delbar p) \cr
&= 2 S\left(\frac{\partial n}{\partial t} + \frac{\partial p}{\partial t},
\frac{\partial n}{\partial \tbar} + \frac{\partial p}{\partial \tbar}\right)
dt \wedge d\tbar \cr
&= 2S(v(t),\overline{v(t)}) dt \wedge d\tbar
\end{align*}
where $v(t) := {\partial n}/{\partial t} + {\partial p}/{\partial t}$. Set $t =
x + i y$. Since $v(t) \in H^{-1,0}_t$ for all $t$, it follows that $\nu^\ast
\w_S = 2i^{-1-0}S(v,\vbar) dx\wedge dy$, which is non-negative.
\end{proof}

\begin{corollary}
\label{cor:positivity}
Suppose that $\V$ is a variation of Hodge structure of weight $-1$ over a smooth
complex algebraic variety $X$. If $S$ is a weak polarization of $\V$ and $\nu$
is a normal function section of $J(\V) \to X$, then for all complete curves $T$
in
$X$
$$
\int_T \nu^\ast \w_S \ge 0
$$
with equality if and only if the infinitesimal invariant of the normal function
vanishes on $T$.
\end{corollary}

Suppose that $\Xbar$ is a smooth completion of $X$. If
Conjecture~\ref{conj:integrable} holds, then we can conclude that the natural
extension of $\nu^\ast\phi_S$ to a class in $H^2(\Xbar)$ has non-negative degree
on all complete curves $T$ in $\Xbar$ that do not lie in $\Xbar-X$.

\section{Slope Inequalities}
\label{sec:moriwaki}

As an immediate consequence of Corollary~\ref{cor:positivity} with the
computations in Section~\ref{sec:formulas}, we obtain the following versions
Moriwaki's inequalities \cite{moriwaki,moriwaki:new}. The second assertion below
was obtained independently by the author in the late 1990s (cf.\
\cite[p.~195]{moriwaki:new}).

\begin{theorem}
\label{thm:moriwaki_weak}
If $g\ge 2$, then
\begin{enumerate}
\item the divisor
$$
(8g+4)\lambda_1 - 4\sum_{h=1}^{[g/2]} h(g-h)\d_h
$$
has non-negative degree on each complete curve in $\M^c_g$,

\item the divisors
\begin{align*}
4g(g-1)\psi - \kappa_1 - \sum_{h=1}^{g-1} (2h-1)^2\delta_{g-h}^{\{x\}} \cr
8\lambda_1 + 4g\psi - \d_0 - \sum_{h=1}^{g-1} h\delta_{g-h}^{\{x\}}
\in H^2(\M_{g,1}^c)
\end{align*}
have non-negative degree on each complete curve in $\M_{g,1}^c$.

\end{enumerate}
\end{theorem}

The semi-positivity of the $2$-forms representing these classes implies that
their powers are also semi-positive.

\begin{corollary} If $g\ge 2$ and $k\ge 1$, then
\begin{enumerate}

\item the cohomology class
$$
\Big((8g+4)\lambda_1 - 4\sum_{h=1}^{[g/2]} h(g-h)\d_h\Big)^k
\in H^{2k}(\M^c_g)
$$
has non-negative degree on each complete $k$-dimensional subvariety of
$\M^c_g$.

\item the cohomology classes
$$
\Big(4g(g-1)\psi - \kappa_1 - \sum_{h=1}^{g-1} (2h-1)^2\delta_h^{\{x\}}\Big)^k
$$
and
$$
\Big(8\lambda_1 + 4g\psi - \d_0 - \sum_{h=1}^{g-1} h\delta_h^{\{x\}}\Big)^k
$$
in $H^{2k}(\M^c_{g,1})$ have non-negative degree on each complete
$k$-dimensional subvariety of $\M^c_{g,1}$.
\end{enumerate}
\end{corollary}

This result also follows from Kleiman's criterion \cite[Thm.~1]{kleiman}, as
J\'anos Koll\'ar pointed out to me.

The statements in Theorem~\ref{thm:moriwaki_weak} are weaker than Moriwaki's in
the sense that they apply only to complete curves in $\M_{g,n}^c$ where $n=0,1$,
but stronger than Moriwaki's as his versions apply only to complete curves in
$\Mbar_{g,n}$ that do not lie in the boundary divisor $\D$. These two versions
suggest the following stronger version of Moriwaki's inequalities, which would
follow from Conjecture~\ref{conj:integrable} if it were true.

\begin{conjecture}
\label{conj:moriwaki} For all $g\ge 2$:
\begin{enumerate}
\item the divisor
$$
M := (8g+4)\lambda_1 - 4\sum_{h=1}^{[g/2]} h(g-h)\d_h
$$
has non-negative degree on each complete curve in $\Mbar_g$ that does not
lie in $\D_0$;

\item the divisors
\begin{align*}
W_H := 4g(g-1)\psi - \kappa_1 - \sum_{h=1}^{g-1} (2h-1)^2\delta_{g-h}^{\{x\}}\cr
W_L := 8\lambda_1 + 4g\psi - \d_0 - \sum_{h=1}^{g-1} h\delta_{g-h}^{\{x\}}
\in H^2(\M_{g,1}^c)
\end{align*}
have non-negative degree on each complete curve in $\Mbar_{g,1}$ that does
not lie in $\D_0$;

\item and for all $\dd = (d_1,\dots,d_n) \in \Z^d$ with $\sum d_j = m$, the
class
\begin{multline*}
-{m^2}\kappa_1/2 +
\sum_{j=1}^n (d_jm+d_j^2/2)\,\psi_j - \sum_{P\subseteq I}
\sum_{\{x_j,x_k\}\subseteq P} d_j d_k\, \d_0^P \cr
-\frac{1}{4} \sum_{P\subseteq I}\sum_{h=1}^{g-1}
\big(d_P - (2h-1)m\big)^2\, \d_h^P
\end{multline*}
has non-negative degree on all complete curves in $\Mbar_{g,n}$ that do not lie
in $\D_0$.

\end{enumerate}
\end{conjecture}

Other slope inequalities of this type may be proved more directly, rather than
deducing them from the fundamental slope inequalities above. For example, since
$F_\dd$ is a normal function, the class $F_\dd^\ast \phi_H$ is represented by a
non-negative $(1,1)$-form on $\M_{g,n}^c$. Theorem~\ref{pullback_phi} thus
implies the following positivity statement:

\begin{proposition}
If $g\ge 2$ and $(d_1,\dots,d_n)\in \Z^n$ satisfies $\sum_{j=1}^n d_j =
(2g-2)m$, then for each $k\ge 1$,
\begin{multline*}
\bigg(-{m^2}\kappa_1/2 +
\sum_{j=1}^n (d_jm+d_j^2/2)\,\psi_j - \sum_{P\subseteq I}
\sum_{\{x_j,x_k\}\subseteq P} d_j d_k\, \d_0^P \cr
-\frac{1}{4} \sum_{P\subseteq I}\sum_{h=1}^{g-1}
\big(d_P - (2h-1)m\big)^2\, \d_h^P\bigg)^k
\end{multline*}
has non-negative degree on all complete $k$-dimensional subvarieties of
$\M_{g,n}^c$.
\end{proposition}

When $k=g$ and $m=g$, this implies that the pullback $F_\dd^\ast \eta_g$ of the
zero section has non-negative degree on all complete, $g$-dimensional
subvarieties of $\M_{g,n}^c$.

\subsection{The jumping divisor}
\label{sec:jump_div}

To better understand the behaviour of the degree of the Moriwaki divisors on
curves in $\Mbar_{g,n}$ that pass through $\D_0^\sing$ but are not contained in
$\D_0$, we need to consider  the phenomenon of ``height jumping'' and the
associated notion of a {\em jumping divisor}.

Suppose that $X=\Xbar-D$ where $\Xbar$ is a smooth projective variety and $D$ is
a normal crossings divisor in $\Xbar$. Suppose that $\U$ is a weight $-1$
variation of Hodge structure over $X$ that is polarized by $S$. Suppose that
$\nu$ is a normal function section of $J(\U)$. Then one has the symmetric
biextension line bundle $\nu^\ast\Bhat$ over $X$. Lear's Theorem implies that a
positive power $\Bhat_{N,\nu}$ of $\nu^\ast\Bhat$ extends naturally to a line
bundle over $\Xbar$. It is characterized by the property that the metric on
$\nu^\ast\Bhat$ extends to a continuous metric on this line bundle over
$\Xbar-D^\sing$. For clarity of exposition, we suppose that the power is $1$, so
that $\nu^\ast\Bhat$ itself extends.\footnote{Otherwise, replace it by the power
that does extend. In all known examples the power is $1$.} Denote the extended
line bundle $\Bhat_{1,\nu}$ by $\B_\Xbar$.

Now suppose that $f : T \to \Xbar$ is a morphism from a smooth projective curve
to $\Xbar$ whose image is not contained in $D$. Set $T'=T-f^{-1}(D)$. Applying
Lear's Theorem to the normal function section $f^\ast\nu$ of $J(f^\ast \U) \to
T'$, we obtain the Lear extension $\B_T$ of $(\nu\circ f)^\ast\Bhat$ to $T$. If
the image of $T$ avoids $D^\sing$, then $f^\ast \B_\Xbar \cong \B_T$. This is
because $\B_\Xbar$ is metrized over $\Xbar - D^\sing$, so that its pullback to
$T$ is the unique metrized extension to $T$ of $(\nu\circ f)^\ast\Bhat$. In
general, there is a $0$-cycle $J$ on $T$, supported on $T-T'$, such that
$$
f^\ast\B_\Xbar \cong \B_T(J).
$$
If $f(T)\cap D^\sing$ is empty, then $J=0$. We will call $J$ the {\em jumping
divisor of $\nu$ on $T$}. It is not always trivial, as we shall explain below.
The jumping divisor encodes ``height jumping''.

\subsection{Height jumping}
Set $d=\dim X$. Assume that $\bD^d$ is a polydisk in $\Xbar$ with coordinates
$(t_1,\dots,t_d)$, and that its intersection with $D$ is the divisor $t_1 \dots
t_m = 0$. Assume that the monodromy of $\U$ about the branch $t_j = 0$ of $D$ is
unipotent for each $j\in \{1,\dots,m\}$.\footnote{This condition is satisfied by
the variations $\H$, $\bL$ and $\V$ over $\M_{g,n}$.} Suppose that
$$
\beta \in H^0\big((\bD^\ast)^m\times \bD^{d-m},\nu^\ast\Bhat\big)
$$
is a biextension section of $\nu^\ast\Bhat$ defined over the punctured polydisk.
The associated {\em height function} $(\bD^\ast)^m\times \bD^{d-m} \to \R^+$
is the function
$$
(t_1,\dots,t_d)\mapsto \log|\beta(t_1,\dots,t_d)|.
$$

Suppose that $P\in T$ and that $f(P)$ is the origin of $\bD^d$. Suppose that
$\bD$ is a disk in $T$ with coordinate $t$ with $t(P)=0$. The restriction of $f$
to $\bD$ is a holomorphic arc $f : \bD \to \bD^d$. Set
$$
r_j := \ord_{t=0} f^\ast t_j,\quad j = 1,\dots, m.
$$
There is a rational number $q(r_1,\dots,r_m)$, which depends only on the
exponents $r_j$, such that
$$
\log|\beta(f(t))|_\B \sim q(r_1,\dots,r_m) \log|t|.
$$
One might expect that $q(r_1,\dots,r_m)$ is linear. Surprisingly, this is not
the case.

To better understand this, write
$$
q(r_1,\dots,r_m) = q_0(r_1,\dots,r_m) + j(r_1,\dots,r_m)
$$
where $q_0$ is the linear function
$$
q_0(r_1,\dots,r_m) = \sum_{j=1}^m r_j q(\ee_j)
$$
and $\ee_1,\dots,\ee_m$ is the standard basis of $\Z^m$. We shall call $j$ the
{\em jump function} of $\beta$ at $P$. When $j$ vanishes, the height behaves as
expected. Surprisingly, the height can jump. I first observed this when trying
to understand Moriwaki's inequality, as explained in the following example.
Although I was aware of height jumping through this example, I had no
explanation for it. Recently Brosnan and Pearlstein
\cite{brosnan-pearlstein:heights} have given a complete explanation of this
phenomenon.

For the curve $f :T \to \Xbar$ and $P\in T$, define $j_P= j(r_1,\dots,r_m)$.
The jumping function determines the jumping divisor.

\begin{proposition}
The jumping divisor $J$ of $\nu$ on $T$ associated to $f:T\to \Xbar$ is the
$0$-cycle
$$
J = \sum_{P\in T} j_P P
$$
on $T$.
\end{proposition}

This is proved, using techniques similar to those described in
\cite[\S8]{hain-reed:arakelov}, by considering the asymptotics as one approaches
$P$ of the length
$$
|s|_\Xbar/|s|_T
$$
of a section $s$ of $f^\ast \B_\Xbar\otimes \B_T^{-1}$ that trivializes it over
$T'$. Here $|\blank|_\Xbar$ denote the natural metric of $(\nu\circ f)^\ast
\Bhat$ over $T'$ and $|\blank|_T$ denotes the natural metric on $\B_T$.

\subsection{An example of height jumping}

In this example, $g\ge 3$, $X = \M_g^c$, $\U$ is the variation $\V$ and $\nu$ is
the normal function associated to the Ceresa cycle. Denote the Lear extension of
$\nu^\ast\Bhat$ to $\Mbar_g$ by $\B_\Mbar$. The main result of
\cite{hain-reed:arakelov} implies that $c_1(\B_\Mbar)$ is the Moriwaki divisor
$$
M:= (8g+4)\lambda_1 - g\d_0 - \sum_{h=1}^{\lfloor g/2\rfloor} 4h(g-h)\d_h
$$

Let $\Hbar_g :=\{[C] \in \Mbar_{g} : C \text{ is hyperelliptic}\}$ be the locus
of hyperelliptic curves in $\Mbar_g$. Set $\cH^c_g = \Hbar_g \cap \M_g^c$. The
normal function $\nu$ vanishes identically on $\cH_g^c$. This implies that the
line bundle $\nu^\ast\Bhat$ over $\cH^c_g$ is trivial as metrized holomorphic
line bundle. Consequently, its Lear extension to $\Hbar_g$, which we denote by
$\B_\Hbar$, is a trivial as a metrized holomorphic line bundle over $\Hbar_g$.

Recall that the boundary $\Hbar_g - \cH_g$ of $\Hbar_g$ is a union of divisors
$$
\D_h,\quad 1\le h \le g/2 \text{ and } \Xi_h,\quad 0\le h \le (g-1)/2,
$$
where $\D_h$ is the restriction of the boundary divisor $\D_h$ of $\Mbar_g$ to
$\Hbar_g$; where $\Xi_0$ is the divisor whose generic point is an irreducible,
geometrically connected hyperelliptic curve with one node; and where $\Xi_h$,
($h\neq 0$) is the locus whose generic point is a hyperelliptic curve with two
nodes that are conjugate under the hyperelliptic involution and whose
normalization is the union of two smooth, geometrically connected hyperelliptic
curves, one of genus $h$ and the other of genus $g-h-1$. Denote the class of
$\Xi_h$ in $H^2(\Hbar_g)$ by $\xi_h$.

Since $\B_\Xbar$ is metrized over $\Mbar_g-\D_0^\sing$, the restrictions of
$\B_\Mbar$ and $\B_\Hbar$ to
$$
\Hbar_g - \bigcup_{h>0} \Xi_h
$$
are isomorphic as metrized line bundles. This implies that
$$
j^\ast \B_\Mbar \otimes \B_\Hbar^{-1} \cong \O(J)
$$
where $j$ denotes the inclusion $\Hbar_g \hookrightarrow \Mbar_g$ and where $J$
is a linear combination of the $\Xi_h$, $h>0$. Note that, since $\B_\Hbar$ is
trivial,
$$
\O(J) \cong j^\ast \B_\Mbar \cong \O(M)|_{\Hbar_g} \in \Pic \Hbar_g.
$$
The restriction of the Moriwaki divisor to $\Hbar_g$ is easily seen to be
$$
j^\ast M =
(8g+4)\lambda_1 - g\xi_0 - \sum_{h=1}^{\lfloor(g-1)/2 \rfloor} 2g \xi_h
- 4\sum_{h=1}^{\lfloor g/2\rfloor} h(g-h)\d_h.
$$
On the other hand, Cornalba and Harris \cite{cornalba-harris} have shown that
$$
(8g+4)\lambda_1 - g\xi_0 - \sum_{h=1}^{\lfloor(g-1)/2 \rfloor} 2(h+1)(g-h)\xi_h 
- 4\sum_{h=1}^{\lfloor g/2\rfloor} h(g-h)\d_h = 0
$$
in $\Pic \Hbar_g$. Together these imply that
$$
J = \sum_{h=1}^{\lfloor(g-1)/2 \rfloor} 2h(g-h-1)\Xi_h.
$$

It is now easy to construct an example of height jumping. Suppose that $f : T
\to \Mbar_g$ is a curve whose image lies in the hyperelliptic locus and is not
contained in $\D_0$. If $h>0$ and if the image of $f$ intersects $\Xi_h$
transversely at a smooth point $f(P)$, then the computations above imply that
$$
j_P = 2(h+1)(g-h)-2g = 2h(g-h-1) > 0.
$$
This implies that
$$
\deg_T M = \deg_T J > 0.
$$

Moriwaki's inequality and positivity in Hodge theory suggest that the jumping
divisor associated to any curve in $\Mbar_g$ should be effective. Denote by
$J_T$ the jumping divisor associated to a morphism $f:T \to \Mbar_g$ whose image
is not contained in $\D_0$.

\begin{conjecture}[weak form]
\label{conj:jump_weak}
For all projective curves $f: T\to \Mbar$ whose image is not contained in
$\D_0$, the jumping divisor $J_T$ is effective.
\end{conjecture}

This and Conjecture~\ref{conj:integrable}, if true, imply a stronger version of
Moriwaki's inequalities as, for example,
$$
\deg_T M = \deg_T \B_T + \deg_T J_T \ge \deg_T J_T \ge 0.
$$
This would imply that the degree of Moriwaki's divisor on most curves not
contained in $\D_0$ that pass through $\D_0^\sing$ would be strictly positive as
$\deg_T M$ would be bounded below by the degree of its jumping divisor.

Similarly, one can conjecture that for all projective curves $f: T \to
\Mbar_{g,1}$ whose image is not contained in $\D_0$, the jumping divisor
associated to the biextension line bundle associated to the  normal function
section $\K$ of $J(\H)$ is always effective.

In general, one might hope that in the situation described in
Paragraph~\ref{sec:jump_div}, the jumping divisor $J$ associated to a curve $f:
T \to \Xbar$ whose image is not contained in $D$, is effective.

\appendix

\section{Normal Functions over $\M_{g,n}$}
\label{sec:vmhs}

For completeness, we state the classification (mod torsion) of normal functions
over $\M_{g,n}$ associated to variations of Hodge structure whose monodromy
representation factors through a rational representation of $\Sp_g$. It follows
quite directly from results proved in \cite[\S8]{hain:normal}.\footnote{Note
that there are two typos on page~121. Line~4 should begin $\dim \Gamma
H^1(\G_{g,r}^n,V(\lambda))$, and there is a $2$ missing from the right-hand side
of line $-7$.}

The isomorphism classes of irreducible rational representations of the
$\Q$-group $\Sp_g$ are indexed by partitions $\lambda$
$$
n = \lambda_1 + \dots + \lambda_h,\quad \lambda_1\ge \lambda_2 \ge \dots\ge
\lambda_h, \quad h\le g
$$
of integers $n$ into at most $g$ parts. Denote the local system over $\A_g$ that
corresponds to the partition $\lambda$ by $\V_\lambda$. It underlies a
$\Q$-variation of Hodge structure of weight $-|\lambda|$, where
$$
|\lambda| := \lambda_1 + \dots + \lambda_h.
$$
These can be pulled back to variations of Hodge structure over $\M_{g,n}$ along
the period map. Note that $\H = \V_{[1]}$ and that $\V=\V_{[1^3]}(-1)$.

Recall that $\G A$ denotes the set of rational $(0,0)$ classes of a $\Q$-Hodge
structure $A$. If $A$ is polarized, then $A = \G A \oplus (\G A)^\perp$ in the
category of $\Q$-Hodge structures.

\begin{theorem}
\label{thm:vmhs}
Suppose that $2g-2+n>0$. If $\U$ is a polarized variation of Hodge structure
over $\M_{g,n}$ whose monodromy representation factors through a rational
representation of $\Sp_g$, then there is an isomorphism of variations of
$\Q$-Hodge structure
$$
\U \cong \bigoplus_\lambda A_\lambda \otimes_\Q \V_\lambda,
$$
where $A_\lambda$ is the Hodge structure $H^0(\M_{g,n},\Hom(\V_\lambda,\U))$.
Moreover, if $A$ is a polarized Hodge structure, then
\begin{enumerate}

\item $\Ext^1_{\hodge(\M_{g,n})}(\Q(0),A \otimes \V_\lambda)$
vanishes unless $\lambda=[1]$ or $\lambda=[1^3]$,

\item $\Ext^1_{\hodge(\M_{g,n})}(\Q(0),A \otimes \H)$ vanishes when $\G A=0$.

\item $\Ext^1_{\hodge(\M_{g,n})}(\Q(0),\H)$ has basis $\K_1,\dots,\K_n$,

\item $\Ext^1_{\hodge(\M_{g,n})}(\Q(0),A \otimes \V)$ vanishes when $\G A=0$,

\item $\Ext^1_{\hodge(\M_{g,n})}(\Q(0), \V)$ is one-dimensional, spanned by
$\nu$.

\end{enumerate}
\end{theorem}

Crudely stated, this result says that all normal functions over $\M_{g,n}$
associated to variations of Hodge structure that are representations of $\Sp_g$
can be expressed, modulo torsion, as rational linear combinations of the basic
normal functions $\nu$ and $\K_1,\dots,\K_n$. For example, $(2g-2)\delta_{j,k} =
\K_j - \K_k$.

\section{The Big Picture}
\label{sec:big_picture}

The philosophy behind this work is that a significant amount of the geometry of
$\Mbar_{g,n}$ is encoded in the category $\hodge_{g,n}$ of those admissible
variations of $\Z$-MHS over $\Mbar_{g,n}$ whose weight graded quotients are
subquotients of Tate twists $\H^{\otimes m}(d)$ of tensor powers of the
fundamental local system
$$
\H := R^1\pi_\ast \Z(1)
$$
associated to the universal curve $\pi : \cC \to \M_{g,n}$. The normal functions
discussed in this paper are objects of $\hodge_{g,n}$.

The pure objects of $\hodge_{g,n}$ are the variations of Hodge structure that
correspond to the irreducible representations of the group $\GSp_g$ of
symplectic similitudes.\footnote{This is well known. An explanation can be found
in \cite[\S8.1]{hain:rat_pts}.} These are all pulled back from variations of
Hodge structure over $\A_g$ along the period mapping $\M_{g,n} \to \A_g$.
Theorem~\ref{thm:vmhs} implies that the only non-trivial extensions between
these (mod torsion) are of the form
\begin{equation}
\label{eqn:extension}
0 \to \V_b \to \E \to \V_a \to 0
\end{equation}
where the weights $w_a$ and $w_b$ of the pure variations $\V_a$ and $\V_b$
satisfy $w_a = 1 + w_a$. This is because the extension (\ref{eqn:extension}) is
determined by the extension
$$
0 \to \Hom_\Z(\V_a,\V_b) \to \E' \to \Z_{\M_{g,n}}(0) \to 0
$$
obtained by applying $\Hom_\Z(\V_a,\blank)$ to (\ref{eqn:extension}) and then
pulling back along the identity $\Z(0) \to \Hom_\Z(\V_a,\V_a)$. That is, the
$1$-extensions in $\hodge_{g,n}$ correspond to normal functions (mod torsion)
over $\M_{g,n}$.

The question arises as to how one can understand $\hodge_{g,n}$. A first
approximation is to understand the monodromy representations of objects of
$\hodge_{g,n}$. The (orbifold) fundamental group of $\M_{g,n}$ is the mapping
class group:
$$
\pi_1(\M_{g,n},[C;P]) \cong \G_{C,P} := \pi_0\Diff^+(C,P),
$$
where $P=\{x_1,\dots,x_n\}$ and $\pi_0\Diff^+(C,P)$ denotes the group of
connected components of the group of orientation preserving diffeomorphisms of
$C$ that fix $P$ pointwise.\footnote{This group depends only on $g$ and $n$ and 
is often denoted by $\G_{g,n}$.} The action of $\G_{C,P}$ on $H_1(C,\Z)$ induces
a homomorphism
$$
\rho : \G_{C,P} \to \Aut\big(H_1(C,\Z),\text{intersection form}\big)
=: \Sp(H_1(C,\Z)) \cong \Sp_g(\Z)
$$
which is well known to be surjective. Its kernel is the {\em Torelli
group}, $T_{C,P}$.

Because each object $\V$ of $\hodge_{g,n}$ is filtered by its weight filtration
$$
\cdots \subseteq W_{j-1}\V \subseteq W_j \V \subseteq W_{j+1}\V \subseteq \cdots
$$
where each $\Gr^W_j \V := W_j\V/W_{j-1}\V$ is a pure variation of Hodge
structure,
the Zariski closure (over $\Q$) of its monodromy representation
$$
\rho_V : \G_{C,P} \to \Aut(V_{C,P})
$$
is an extension
$$
1 \to U_V \to G_V \to \Sp(H_1(C)) \to 1
$$
of algebraic $\Q$-groups where $U_V$ is unipotent. It is thus natural to
consider all Zariski dense representations $\G_{C,P} \to G(\Q)$ where $G$ is a
$\Q$ algebraic group that is an extension of $\Sp(H_1(C))$ by a unipotent group
and where $\G_{C,P} \to G \to \Sp(H_1(C))$ is the standard representation
$\rho$. These form an inverse system. Their inverse limit is known as the {\em
completion of $\G_{C,P}$ relative to $\rho$}; it is studied in
\cite{hain:torelli}.

The completion of $\G_{C,P}$ relative to $\rho$ is an extension
$$
1 \to \cU_{C,P} \to \cG_{C,P} \to \Sp(H_1(C)) \to 1
$$
of affine $\Q$-groups, where $\cU_{C,P}$ is prounipotent. There is a canonical
homomorphism $\G_{C,P} \to \cG_{C,P}(\Q)$. Denote the corresponding sequence of
Lie algebras by
\begin{equation}
\label{eqn:ses}
0 \to \u_{C,P} \to \g_{C,P} \to \sp(H_1(C)) \to 0.
\end{equation}
It is proved in \cite{hain:malcev} and \cite{hain:torelli} that for each choice
of a base point $[C,P]$ of $\M_{g,n}$, the sequence (\ref{eqn:ses}) is a
short exact sequence of Lie algebras in $\hodge$.

Define a {\em Hodge representation} of $\G_{C,P}$ to be a MHS $V$ and a pair of
representations $\phi_\Z : \G_{C,P} \to \Aut_\Z V$ and $\phi : \cG_{C,P} \to
\Aut V$ such that  the diagram
$$
\xymatrix{
\G_{C,P} \ar[r]^{\phi_\Z}\ar[d] & \Aut_\Z V \ar[d] \cr
\cG_{C,P}(\Q) \ar[r]^\phi & \Aut_\Q V
}
$$
commutes and the induced homomorphism $d\phi : \g_{C,P} \to \End V$ is a
morphism of MHS. The main result of \cite{vmhs} implies that $\hodge_{g,n}$ is
equivalent to the category of Hodge representations of $\g_{C,P}$, a statement
informally conjectured by Deligne. Because of this, extensions in $\hodge_{g,n}$
are determined up to isogeny by Lie algebra cohomology via the following
isomorphisms (cf.\ \cite[Cor.~3.7]{hain:nab}):
$$
\Ext^\dot_{\hodge_{g,n}}(\Q(0),\V) \cong H^\dot(\cG_{C,P},V_{C,P})
\cong H^\dot(\u_{C,P},V_{C,P})^{\sp(H_1(C))}.
$$
Although Theorem~\ref{thm:vmhs} was originally proved by a direct argument, it
is most natural to regard it as a consequence of this general result and known
facts about the cohomology of mapping class and Torelli groups.

The relevance of the algebra $A_{g,n}^\dot$ defined in the introduction is that
there is a natural $\Sp_g$-equivariant algebra homomorphism
$$
H^\dot(\u_{C,P}) \to A_{g,n}^\dot
$$
which is an isomorphism in degrees $0$ and $1$ and injective in degree $2$. The
construction is described in \cite{hain:malcev} in a more general context. The
algebra $T_{g,n}^\dot$ is simply the subalgebra of $A_{g,n}^\dot$ generated by
the image of $H^1(\u)$ --- that is, by normal functions.

The coefficients of the boundary component $\d_h^P$, ($h>0$) in the formulas in
Section~\ref{sec:formulas} are determined by the image in the second weight
graded quotient of $\u_{C,P}$ of the Dehn twist corresponding to a loop about
$\D_h^P$. The coefficients of $\d_0^P$ can be determined similarly, although one
has to work with the appropriate relative weight filtration.

One remaining question is whether more information can be obtained by
considering other natural categories of MHS over $\M_{g,n}$ (possibly with a
level structure), such as the category of variations of MHS obtained from the
Prym construction. At present not enough is know about the topology of the Prym
construction to understand this problem, although recent progress has been
made by Putman \cite{putman}.

\end{document}